\documentclass[final,onefignum,onetabnum]{siamart190516}

\usepackage{braket,amsfonts}

\usepackage{array}

\usepackage[caption=false]{subfig}
\usepackage{pgfplots}

\newsiamthm{claim}{Claim}
\newsiamremark{remark}{Remark}
\newsiamremark{hypothesis}{Hypothesis}
\crefname{hypothesis}{Hypothesis}{Hypotheses}

\usepackage{algorithmic}

\usepackage{graphicx,epstopdf}

\Crefname{ALC@unique}{Line}{Lines}

\usepackage{amsopn}

\usepackage{xspace}
\usepackage{bold-extra}
\usepackage[most]{tcolorbox}

\colorlet{texcscolor}{blue!50!black}
\colorlet{texemcolor}{red!70!black}
\colorlet{texpreamble}{red!70!black}
\colorlet{codebackground}{black!25!white!25}

\lstdefinestyle{siamlatex}{%
  style=tcblatex,
  texcsstyle=*\color{texcscolor},
  texcsstyle=[2]\color{texemcolor},
  keywordstyle=[2]\color{texemcolor},
  moretexcs={cref,Cref,maketitle,mathcal,text,headers,email,url},
}

\tcbset{%
  colframe=black!75!white!75,
  coltitle=white,
  colback=codebackground, 
  colbacklower=white, 
  fonttitle=\bfseries,
  arc=0pt,outer arc=0pt,
  top=1pt,bottom=1pt,left=1mm,right=1mm,middle=1mm,boxsep=1mm,
  leftrule=0.3mm,rightrule=0.3mm,toprule=0.3mm,bottomrule=0.3mm,
  listing options={style=siamlatex}
}

\newtcblisting[use counter=example]{example}[2][]{%
  title={Example~\thetcbcounter: #2},#1}

\newtcbinputlisting[use counter=example]{\examplefile}[3][]{%
  title={Example~\thetcbcounter: #2},listing file={#3},#1}

\DeclareTotalTCBox{\code}{ v O{} }
{ 
  fontupper=\ttfamily\color{black},
  nobeforeafter,
  tcbox raise base,
  colback=codebackground,colframe=white,
  top=0pt,bottom=0pt,left=0mm,right=0mm,
  leftrule=0pt,rightrule=0pt,toprule=0mm,bottomrule=0mm,
  boxsep=0.5mm,
  #2}{#1}

\patchcmd\newpage{\vfil}{}{}{}
\usepackage[active]{srcltx}
\usepackage{MnSymbol}
\usepackage{mathrsfs,color}
\newcommand{\blue}{\color{blue} }

\def\veps{\varepsilon}
\newcommand{\la}{\ensuremath{\lambda}}
\newcommand{\bF}{\mathbb{F}}
\newcommand{\bC}{\mathbb{C}}
\newcommand{\cK}{{\cal K}}
\newcommand{\cB}{{\cal B}}

\DeclareMathOperator{\rank}{rank}
\DeclareMathOperator{\cod}{cod}
\DeclareMathOperator{\diag}{diag}
\DeclareMathOperator{\orb}{O}

\DeclareMathOperator{\GSYL}{GSYL}
\DeclareMathOperator{\bun}{B}
\DeclareMathOperator{\rev}{rev}
\DeclareMathOperator{\POL}{POL}
\DeclareMathOperator{\PEN}{PENCIL}

\newcommand{\ddd}{
\text{\begin{picture}(12,8)
\put(-2,-4){$\cdot$}
\put(3,0){$\cdot$}
\put(8,4){$\cdot$}
\end{picture}}}

\newcommand{\sdotsss}%
{\text{\raisebox{-2.2pt}{$\cdot\,$}%
 \raisebox{1.7pt}{$\cdot$}%
\raisebox{5.6pt}{$\,\cdot$}}}

\renewcommand{\le}{\leqslant}
\renewcommand{\ge}{\geqslant}

\newcommand{\hide}[1]{}

\title{Generic symmetric matrix polynomials with bounded rank and fixed odd grade\thanks{Submitted to the editors DATE.\funding{Partially supported by Ministerio de Econom\'{i}a y Competitividad of Spain
through grant MTM2015-65798-P, by Ministerio de Ciencia, Innovaci\'{o}n y Universidades of Spain through grant MTM2017--90682--REDT (F. De Ter\'{a}n and F. M. Dopico), and by the Swedish Foundation for International Cooperation in Research and Higher Education through grant STINT IB2018-7538 (A. Dmytryshyn).}}}

\date{}

\author{Fernando De Ter\'an\thanks{Departamento de Matem\'aticas, Universidad Carlos III de Madrid, Avenida de la Universidad 30, 28911, Legan\'es, Spain.
(\email{fteran@math.uc3m.es, dopico@math.uc3m.es)}} \and Andrii Dmytryshyn\thanks{School of Science and Technology, \"Orebro University, \"Orebro, SE-70182, Sweden.
(\email{andrii.dmytryshyn@oru.se)}} \and Froil\'an M. Dopico\footnotemark[2]}

\begin{document}
\maketitle

\begin{abstract} We determine the generic complete eigenstructures for $n \times n$ complex symmetric matrix polynomials of odd grade $d$ and rank at most $r$. More precisely, we show that the set of $n \times n$ complex symmetric matrix polynomials of odd grade $d$, i.e., of degree at most $d$, and rank at most $r$ is the union of the closures of the $\lfloor rd/2\rfloor+1$ sets of symmetric matrix polynomials having certain, explicitly described, complete eigenstructures. Then, we prove that these sets are open in the set of $n \times n$ complex symmetric matrix polynomials of odd grade $d$ and rank at most $r$. In order to prove the previous results, we need to derive necessary and sufficient conditions for the existence of symmetric matrix polynomials with prescribed grade, rank, and complete eigenstructure, in the case where all their elementary divisors are different from each other and of degree $1$. An important remark on the results of this paper is that the generic eigenstructures identified in this work are completely different from the ones identified in previous works for unstructured and skew-symmetric matrix polynomials with bounded rank and fixed grade larger than one, because the symmetric ones include eigenvalues while the others not. This difference requires to use new techniques.
\end{abstract}


\section{Introduction}\label{sect.intro}

We deal in this paper with symmetric $n\times n$ matrix polynomials with complex coefficients, i.e.,
\begin{equation}\label{poly}
P(\lambda) = \lambda^{d}A_{d} + \dots +  \lambda A_1 + A_0,
\quad A_i^\top=A_i, \ A_i \in \mathbb C^{n \times n},\  \ \text{for } i=0, \dots, d,
\end{equation}
where $A^\top$ denotes the transpose of the matrix $A$. The matrix polynomial \eqref{poly} is said to have {\em grade} $d$, since the coefficient $A_d$ is allowed to be zero. Thus, the grade $d$ may be larger than the {\em degree} of \eqref{poly}, which is the largest index $i$ such that $A_i \ne 0$.

Symmetric matrix polynomials arise in a variety of applications, like the vibration analysis of mechanical systems and signal processing \cite{TiMe01}. The most relevant case in applications is the quadratic case ($d=2$) with real coefficients, see, for instance, \cite{BHMS13,hmmt,lambda-matrices,TiMe01}, though symmetric polynomials of higher degree, and with complex non-real coefficients have been also encountered in applications \cite{BHMS13}.

The relevant information of matrix polynomials arising in applications is typically encoded in their {\em eigenstructure} (see Section \ref{sect.prempolys} for the definition). The main goal of this paper is to describe the {\em most likely} eigenstructure of symmetric $n\times n$ complex matrix polynomials with odd grade $d$ and bounded deficient rank $r$. As a conclusion of our main result (Theorem \ref{mainth}), we will see that there is no a unique most likely eigenstructure, but several ones, for any given $n$, odd $d$, and $r$. More precisely, the number of most likely different eigenstructures is $\lfloor rd/2\rfloor+1$, where $\lfloor x\rfloor$ denotes the largest integer less than or equal to the real number $x$. A key ingredient in Theorem \ref{mainth} is the notion of {\em bundle}, consisting of the set of $n\times n$ symmetric matrix polynomials with the same grade and having the same eigenstructure, though maybe corresponding to different eigenvalues. Then, in Theorem \ref{mainth}, we describe the set of $n\times n$ complex symmetric matrix polynomials with rank at most $r$ and odd grade $d$ as the union of the closures of $\lfloor rd/2\rfloor+1$ bundles (where the topology used to define ``closures'' is the one induced by a natural Euclidean metric in the space of $n\times n$ symmetric matrix polynomials with grade $d$). Moreover, in Theorem \ref{th:open}, we prove that these bundles are open in the set of $n\times n$ complex symmetric matrix polynomials with rank at most $r$ and odd grade $d$.  The eigenstructures corresponding to these bundles are the ones that we term as {\em generic}. The reason for using this name is that the subset of matrix polynomials with these eigenstructures is open and dense in the set of $n\times n$ complex symmetric matrix polynomials with rank at most $r$ and odd grade $d$.

This work is related to a series of papers that have described the generic eigenstructures of either general (unstructured) matrix pencils (i.e., matrix polynomials of grade $d=1$) and general (unstructured) matrix polynomials with fixed grade $d >1$, or matrix pencils and polynomials with fixed grade $d >1$ with some specific symmetry structures, and in all cases with bounded (deficient) rank. In particular, the generic eigenstructures of general $m\times n$ matrix pencils with rank at most $r<\min\{m,n\}$ were described in \cite{DeDo08}, whereas the ones for $m\times n$ matrix polynomials of fixed grade $d>1$ and rank $r<\min\{m,n\}$ have been recently obtained in \cite{dmydop-laa-2017}\footnote{The case $r = \min\{m,n\}$, with $m\ne n$, was studied in \cite{DeEd95,EdEK99,VanD79} for general pencils and in \cite{dmydop-laa-2017} for general matrix polynomials with fixed grade $d >1$. When $r = \min\{m,n\}$ there is only one generic eigenstructure, in contrast with the $rd + 1$ generic eigenstructures present in the case $r<\min\{m,n\}$. The case $r = \min\{m,n\}$, with $m=n$, is trivial since pencils and matrix polynomials are generically regular with all their eigenvalues different.}. In \cite{DeTe18}, the generic eigenstructures of $\top$-palindromic and $\top$-alternating $n\times n$ matrix pencils with rank $r<n$ were obtained, and \cite{dmydop-laa-2018} presents the generic eigenstructures of $n\times n$ skew-symmetric matrix polynomials of odd grade $d$ and rank $2r<n$. Some of the developments in this last reference rely on results from \cite{DmKa14}, where a complete stratification of the set of skew-symmetric $n\times n$ matrix pencils was obtained. Table \ref{summary-table} summarizes this information. For the sake of completeness, we mention the seminal work \cite{Wate84} (see also \cite{DeEd95}), where the generic eigenstructures for general (unstructured) square singular pencils and singular symmetric and Hermitian pencils were given.

\begin{center}
\begin{table}
\begin{center}
\begin{tabular}{|c||c|c|}\hline
grade&$1$ (pencils)&$d>1$\\\hline
unstructured&\cite{DeDo08}&\cite{dmydop-laa-2017}\\
skew-symmetric&\cite{dmydop-laa-2018}&\cite{dmydop-laa-2018} ($d$ odd)\\
$\top$-palindromic&\cite{DeTe18}&Open\\
$\top$-alternating&\cite{DeTe18}&Open\\
symmetric&\cite{dtdmydop2019}&This work ($d$ odd)\\\hline
\end{tabular}
\caption{Works presenting the generic eigenstructures of matrix pencils/polynomials with bounded (deficient) rank.}\label{summary-table}
\end{center}
\end{table}
\end{center}

In particular, the present work follows, for symmetric matrix polynomials, the approach developed in \cite{dmydop-laa-2018} for skew-symmetric matrix polynomials, though the obtained generic eigenstructures have very different natures in both cases. The approach in \cite{dmydop-laa-2018} requires the knowledge of the generic eigenstructure of skew-symmetric pencils with bounded rank (which is first provided in that reference and is unique), and then, using a particular skew-symmetric linearization (with size $nd\times nd$), the unique generic eigenstructure for $n\times n$ skew-symmetric matrix polynomials of higher odd grade $d$ is obtained. A key point in the argument is to prove that the generic eigenstructure for $nd\times nd$ skew-symmetric matrix pencils with rank at most $n(d-1)+r$ is realizable as the eigenstructure of a skew-symmetric linearization of an $n\times n$ skew-symmetric matrix polynomial with rank at most $r<n$ and grade $d$. The difference in the ranks ($n(d-1)+r$ for the linearization and $r$ for the polynomial) comes from the fact that, if the matrix polynomial $P(\la)$ has rank $r$, then the corresponding linearization of $P(\la)$ will have rank $n(d-1)+r$. A similar approach to this one can now be followed for symmetric polynomials, since, recently, we have obtained the generic eigenstructures for symmetric matrix pencils with bounded rank in \cite{dtdmydop2019}.


Thus, roughly speaking, the key point in our developments consists of proving the existence of symmetric matrix polynomials with rank at most $r$ and odd grade $d$ with complete eigenstructures corresponding (via linearization) to some of the generic eigenstructures of symmetric pencils with bounded rank identified in \cite[Thm. 3.2]{dtdmydop2019}, something that was known in the skew-symmetric case as a consequence of results in \cite{Dmyt15}, but that is open in the symmetric case. The ``some'' in the previous sentence refers to the fact that only the eigenstructures in \cite[Thm. 3.2]{dtdmydop2019} with sufficiently large minimal indices may correspond to matrix polynomials, because of the shifts on the minimal indices in the linearization with respect to those in the polynomial (see Theorem \ref{skewlin}). As a consequence, not all the generic eigenstructures of symmetric matrix pencils with bounded rank are realizable as the eigenstructure of a symmetric linearization of a symmetric matrix polynomial with bounded rank and grade $d$.

The eigenstructures in
\cite[Thm. 3.2]{dtdmydop2019} have two key properties:
\begin{enumerate}
  \item They have eigenvalues, which are all simple (``simple'' means only one elementary divisor of degree $1$ per eigenvalue), and
  \item their right minimal indices are equal to their left minimal indices (imposed by the symmetry) and the right minimal indices differ at most by one.
\end{enumerate}
We prove in this paper that these are, also, properties that characterize the generic complete eigenstructures of $n\times n$ symmetric matrix polynomials with bounded rank $r$ and fixed odd grade $d$. However, the generic symmetric eigenstructures identified in this work are very different from the ones identified in \cite{dmydop-laa-2017,dmydop-laa-2018} for, respectively, general and skew-symmetric matrix polynomials with bounded rank and fixed grade. First, their numbers are different: there are $\lfloor rd/2\rfloor+1$ generic symmetric eigenstructures, $rd +1$ general (unstructured) generic eigenstructures and only one in the skew-symmetric case. Second, and perhaps more important, the generic symmetric eigenstructures include eigenvalues, while the others not. These differences illustrate that considering different classes of structured matrix polynomials has a dramatic effect on the corresponding sets of matrix polynomials with bounded rank and grade. It is precisely the presence of eigenvalues in the generic eigenstructures what leads to the notion of bundle, and introduces another relevant difference between this work and \cite{dmydop-laa-2017,dmydop-laa-2018}. In particular, the main eigenstructures in \cite{dmydop-laa-2017,dmydop-laa-2018} were described using {\em orbits}, which are well-studied objects. However, they are not enough to describe the generic eigenstructures in the present work, given in terms of bundles instead. A bundle is an infinite union of orbits, which results in a more complex object whose topological properties are not so well-known and deserve a more careful treatment.

As an aside result, which has entity by itself, we prove, in Theorem \ref{thm.existence}, a more general inverse result than the one needed to capture the eigenstructures coming from \cite[Thm. 3.2]{dtdmydop2019}: we prove the existence of a symmetric matrix polynomial of rank $r$ and grade $d$ whose complete eigenstructure consists of any list of linear elementary divisors corresponding to simple eigenvalues, any list of right minimal indices, and the same list of left minimal indices, whenever these prescribed elementary divisors and minimal indices satisfy the Index Sum Theorem for $r$ and $d$ \cite[Thm. 6.5]{DeDM14}.

One of the possible applications of the results in this work, as well as the ones in the references from Table \ref{summary-table}, is in the context of low-rank perturbations of matrix pencils and matrix polynomials, when one is interested in analyzing the change of the eigenstructure under such kind of perturbations (see, for instance, \cite{DMM19} and the references therein). The knowledge of the generic eigenstructures may help not only to describe the set of perturbations under consideration but also to understand the changes in the eigenstructure. As can be seen in \cite{DMM19}, there are many references available in the literature that deal with low-rank perturbations of matrix pencils, however we are only aware of the particular results in \cite{DeDo09} for matrix polynomials of fixed grade larger than $1$.

The rest of the paper is organized as follows. In Sections \ref{sect.prempencils} and \ref{sect.prempolys} we recall some basic notions on matrix pencils and matrix polynomials, respectively. Section \ref{sec.existence} presents the construction of a symmetric $n\times n$ matrix polynomial with a given grade and rank having some prescribed lists of linear elementary divisors associated with different eigenvalues and right minimal indices. Section \ref{sect.linz} presents the linearizations used to translate the problem on matrix polynomials into the framework of matrix pencils, as mentioned above, and we describe some relevant features of these linearizations. These properties are key in Section \ref{sec.main}, which is devoted to state and prove the main result of the manuscript, namely the decomposition of the set of $n\times n$ symmetric matrix polynomials with odd grade $d$ and bounded deficient rank in terms of their generic eigenstructures. In Section \ref{sec.openbunbles} we prove that the sets (bundles) of matrix polynomials with the generic eigenstructures are open in the set of $n\times n$ symmetric matrix polynomials with odd grade $d$ and rank at most $r$. This ``openness'' result is considerably hard to prove as a consequence of the fact that the generic eigenstructures include eigenvalues. Finally, in Section \ref{sec.conclusion} we summarize the main contributions of this work and indicate some lines of further research.

\section{Preliminaries on matrix pencils} \label{sect.prempencils}

We start by recalling the Kronecker-type canonical form for symmetric matrix pencils under congruence. In this paper, we consider matrix pencils over the field of complex numbers $\mathbb C$.

For each positive integer $n$ define the $n\times n$ identity matrix $I_n$ and the $n\times n$ symmetric matrix pencils
\begin{align*}
\label{lade}
{\cal J}_n^s(\mu)&:=
\begin{bmatrix}
&&1&\lambda - \mu\\
&\ddd&\ddd&\\
1&\lambda - \mu&&\\
\lambda - \mu&&&
\end{bmatrix}
\quad \text{ and } \quad
{\cal J}_n^s(\infty)&:=
\begin{bmatrix}
&&\lambda&1\\
&\ddd&\ddd&\\
\lambda&1&&\\
1&&&
\end{bmatrix}.
\end{align*}
For $n=1$, we drop the index $s$ and write ${\cal J}_1(\mu):={\cal J}_1^s(\mu)=\lambda-\mu$ and ${\cal J}_1(\infty):={\cal J}_1^s(\infty)=1$. For each nonnegative integer $n$ {we} define the $n\times(n+1)$ matrices
\begin{align*}
F_n& :=
\begin{bmatrix}
1&0&&0\\
&\ddots&\ddots&\\
0&&1&0\\
\end{bmatrix}
\quad \text{ and } \quad
G_n :=
\begin{bmatrix}
0&1&&0\\
&\ddots&\ddots&\\
0&&0&1\\
\end{bmatrix},
\end{align*}
and using them {we} define the $n\times(n+1)$ matrix pencil ${\cal L}_n:=\lambda G_n + F_n$ and the $(2n+1)\times(2n+1)$ symmetric matrix pencil
$$
{\cal M}_n:=
\begin{bmatrix}0&{\cal L}_n^\top\\
{\cal L}_n &0
\end{bmatrix}.$$
\noindent Observe that ${\cal M}_0$ is just the $1 \times 1$ zero matrix pencil.

An $n \times n$ matrix pencil $\lambda A + B$ is {said to be} {\it congruent} to $\lambda C + D$ if there is a nonsingular matrix $S$ such that $S^{\top}AS =C$ and $S^{\top}BS=D$.
In the following theorem we recall the canonical form under congruence of symmetric matrix pencils, i.e., those satisfying $(\lambda A + B)^\top = (\lambda A + B)$.


\begin{theorem}{\rm \cite{Thom91}}\label{kron}
Each $n \times n$ complex symmetric matrix pencil $\lambda A + B$ is congruent
to a direct sum, uniquely determined up
to permutation of summands, of pencils of the forms ${\cal J}_{h}^s(\mu), {\cal J}_{k}^s(\infty)$, and ${\cal M}_m$.
\end{theorem}

In Theorem \ref{kron}, the different blocks ${\cal J}_{h}^s(\mu)$ and ${\cal J}_{k}^s(\infty)$ reveal the eigenstructure corresponding to finite and infinite eigenvalues, respectively, of $\la A +B$. The blocks ${\cal M}_m$ reveal the right (column) and left (row) minimal indices of $\la A+ B$, which are equal to each other in the case of symmetric matrix pencils.

The {\it orbit} of the $n\times n$ symmetric pencil $\lambda A + B$ under the
action of the group $GL_n(\mathbb C)$ on the space of all $n\times n$ symmetric matrix pencils by congruence {is the set of pencils which are congruent to $\la A+B$, namely}:
\begin{equation} \label{equorbit}
\orb^c (\lambda A + B) = \{S^{\top} (\lambda A + B) S \ : \ S \in GL_n(\mathbb C)\}.
\end{equation}
Note that all pencils in $\orb^c (\lambda A + B)$ have the same congruent canonical form, i.e. the same eigenvalues and the same canonical blocks ${\cal J}_{h}^s(\mu), {\cal J}_{k}^s(\infty)$, and the same left (and right) minimal indices or, equivalently, the same blocks ${\cal M}_m$. The {\it congruence bundle} of $\la A+B$, denoted $\bun^{\rm{c}}(\lambda A + B)$, is the union of {all} symmetric matrix pencil orbits under congruence with the same canonical block structure (equal block sizes) as $\la A+B$, but where the distinct eigenvalues (including the infinite one) may take any values as long as they remain distinct, see also \cite{DFKK15,DmJK17,DmKa14,EdEK99}.

The space of all $n\times n$ symmetric matrix pencils is denoted by $\PEN_{n \times n}^s$. Using the Frobenius norm of complex matrices \cite{Highambook}, {we} define a distance in $\PEN_{n \times n}^s$ as $d(\lambda A {+} B, \lambda C {+} D) := \sqrt{\|A-C\|_F^2 + \|B-D\|_F^2}$, which makes $\PEN_{n \times n}^s$ into a metric space. This metric allows us to consider closures of subsets of $\PEN_{n \times n}^s$, in particular, closures of congruence bundles, denoted by $\overline{\bun^c}(\lambda A {+} B)$.

The following theorem provides a description of the generic complete eigenstructures (or canonical forms under congruence) of the set of $n\times n$ symmetric matrix pencils of rank at most $r$, for $r<n$. To be {precise}, this set is presented as the union of the closures of $\lfloor\frac{r}{2}\rfloor+1$ symmetric bundles, which determine the generic eigenstructures.

\begin{theorem}{\rm (Generic eigenstructures of symmetric matrix pencils with bounded rank, \cite[Theorem 3.2]{dtdmydop2019}).} \label{th:improvedDeDo}
Let $n$ and $r$ be integers such that $n \geq 2$ and $1\leq r \leq n-1$.
Define the following $\lfloor\frac{r}{2}\rfloor+1$ symmetric canonical forms of $n\times n$ complex symmetric matrix pencils with rank $r$:
\begin{equation}\label{max}
{\cal
K}_{a} (\lambda):=\diag(\underbrace{{\cal M}_{\alpha+1},\dots,{\cal M}_{\alpha+1}}_{s},
\underbrace{{\cal M}_{\alpha},\dots,{\cal M}_{\alpha}}_{n-r-s}, {\cal J}_1(\mu_1), \dots , {\cal J}_1(\mu_{r-2a}) ),\,
\end{equation}
for $a=0,1,\ldots,\lfloor\frac{r}{2}\rfloor\, $, where $a=(n-r)\alpha+s$ is the Euclidean division of $a$ by $n-r$, and $\mu_1,\dots,\mu_{r-2a}$ are arbitrary complex numbers (different from each other).
Then:
\begin{enumerate}
\item[\rm (i)] For every $n\times n$ symmetric pencil ${\cal S }(\lambda)$ with rank
at most $r$, there exists an integer $a$ such that
$\overline{\bun^c}({\cal K}_{a}) \supseteq\overline{\bun^c}({\cal S})$.
\item[\rm (ii)] $\overline{\bun^c}({\cal K}_{a})
\not\supseteq \overline{\bun^c}({\cal K}_{a'})$ whenever $a \ne
a'$.
\item[\rm (iii)] The set of $n\times n$ symmetric matrix pencils of rank at most $r$ is a closed subset of $\PEN^{s}_{n \times n}$, and it is equal to $\displaystyle \bigcup_{0\leq a \leq \lfloor\frac{r}{2}\rfloor} \overline{\bun^c}({\cal K}_{a} )$.
\end{enumerate}
\end{theorem}

\section{Preliminaries on  matrix polynomials} \label{sect.prempolys}
As mentioned in Section \ref{sect.intro}, we are interested in this paper in $n\times n$ symmetric matrix polynomials of grade $d$, i.e. of degree at most $d$, over $\bC$. However, in this section we consider general (not necessarily symmetric) complex matrix polynomials, since most of the presented results are valid in this general setting. Those which are valid only for symmetric matrix polynomials will be clearly indicated. Moreover, throughout the rest of the paper the terms ``matrix polynomial'' and ``polynomial matrix'' are used with exactly the same meaning.

We start {by} recalling the complete eigenstructure of a matrix polynomial, i.e. the definition of the elementary divisors and minimal indices.
\begin{definition}
Let $P(\lambda)$ and $Q(\lambda)$ be two $m \times n$ matrix polynomials. Then $P(\lambda)$ and $Q(\lambda)$ are {\em unimodularly equivalent} if there exist two unimodular matrix polynomials $U(\lambda)$ and $V(\lambda)$ (i.e., $\det U(\lambda), \det V(\lambda) \in \mathbb C \backslash \{0\}$) such that
$$
U(\lambda) P(\lambda) V(\lambda) = Q(\lambda).
$$
\end{definition}
The transformation $P(\lambda) \mapsto U(\lambda) P(\lambda) V(\lambda)$
is called a {\em unimodular equivalence transformation}
and the canonical form with respect to this transformation is the {\it Smith form} { (see, for instance, \cite{Gant59})}.
\begin{theorem}[Smith form] \label{tsmiths}
Let $P(\lambda)$ be an $m\times n$ matrix polynomial over $\mathbb C$. Then there exist $r \in \mathbb N$, $r \le \min \{ m, n \}$, and {two} unimodular matrix polynomials $U(\lambda)$ and $V(\lambda)$ over $\mathbb C$ such that
\begin{equation} \label{smithform}
U(\lambda) P(\lambda) V(\lambda) =
\left[
\begin{array}{ccc|c}
g_1(\lambda)&&0&\\
&\ddots&&0_{r \times (n-r)}\\
0&&g_r(\lambda)&\\
\hline
&0_{(m-r) \times r}&&0_{(m-r) \times (n-r)}
\end{array}
\right],
\end{equation}
where $g_j(\lambda)$ is a scalar monic polynomial, for $j=1, \dots, r$, and $g_j(\lambda)$ divides $g_{j+1}(\lambda)$, for $j=1, \dots, r-1$. Moreover, the canonical form \eqref{smithform} is unique.
\end{theorem}
The integer $r$ from Theorem \ref{tsmiths} is called the {\it rank} of the matrix polynomial $P(\lambda)$ {(referred to as the {\em normal rank} sometimes in the literature)}. {The polynomials} $g_j(\lambda)$, {for} $j=1, \dots, r,$ are called {the} {\em invariant polynomials} of $P(\lambda)$, and each of them can be uniquely factored as
$$
g_j(\lambda) = (\lambda - \alpha_1)^{\delta_{j1}} \cdot
(\lambda - \alpha_2)^{\delta_{j2}}\cdot \ldots \cdot
(\lambda - \alpha_{l_j})^{\delta_{jl_j}},
$$
where $l_j \ge 0, \ \delta_{j1}, \dots, \delta_{jl_j} > 0$ are integers. If $l_j=0$ then $g_j(\lambda)=1$. The complex numbers $\alpha_1, \dots, \alpha_{l_j}$ are {the} {\em finite eigenvalues} of $P(\lambda)$. The {\it elementary divisors} of $P(\lambda)$ associated with each finite eigenvalue $\alpha_{k}$ is the collection of factors $(\lambda - \alpha_{k})^{\delta_{jk}}$ (possibly with repetitions).

If $P(\la)$ is a matrix polynomial of grade $d$ and zero is an eigenvalue of $\rev P(\lambda):= \lambda^dP(1/\lambda)$, then we say that $\lambda = \infty$ is an eigenvalue of $P(\lambda)$. The elementary divisors $\lambda^{\gamma_k}, \gamma_k > 0,$ for the zero eigenvalue of $\rev P(\lambda)$ are the elementary divisors associated with the infinite eigenvalue of $P(\lambda)$. We emphasize that this definition is applied even in the case {where} the exact degree of $P(\la)$ {\em is less than} $d$.

The left and right null-spaces, over the field of rational functions $\mathbb C(\lambda)$, for an $m\times n$ matrix polynomial $P(\lambda)$, {are defined} as follows:
\begin{align*}
{\cal N}_{\rm left}(P)&:= \{y(\lambda)^\top \in \mathbb C(\lambda)^{1 \times m}: y(\lambda)^\top P(\lambda) = 0_{1\times n} \}, \\
{\cal N}_{\rm right}(P)&:= \{x(\lambda) \in \mathbb C(\lambda)^{n\times 1}: P(\lambda)x(\lambda) = 0_{m\times 1}\}.
\end{align*}
Each rational subspace ${\cal V}$ of $\mathbb C(\lambda)^n$ (where ``rational'' indicates that the underlying field is $\mathbb C(\lambda)$) has bases consisting entirely of vector polynomials. A basis of ${\cal V}$ whose sum of degrees is minimal among all such bases of ${\cal V}$ is called a {\it minimal basis} of ${\cal V}$. The ordered list of degrees of the vector polynomials in any minimal basis of ${\cal V}$ is always the same. These degrees are called the {\em minimal indices} of ${\cal V}$ \cite{Forn75,Kail80}. 
More precisely, let the sets $\{y_1(\lambda)^\top,...,y_{m-r}(\lambda)^\top\}$ and $\{x_1(\lambda),...,x_{n-r}(\lambda)\}$ be minimal bases of ${\cal N}_{\rm left}(P)$ and ${\cal N}_{\rm right}(P)$, respectively, ordered so that $0 \le \deg(y_1) \le \dots \le \deg(y_{m-r})$ and $0\le \deg(x_1) \le \dots \le \deg(x_{n-r})$. Let $ \eta_k = \deg(y_k)$, for $k=1, \dots , m-r$, and $ \varepsilon_k = \deg(x_k)$, for $k=1, \dots , n-r$. Then the scalars $0 \le  \eta_1 \le \eta_2 \le  \dots \le \eta_{m-r}$ and $0 \le  \varepsilon_1 \le \varepsilon_2 \le  \dots \le \varepsilon_{n-r}$ are, respectively, the {\it left} and {\it right minimal indices} of~$P(\lambda)$. Note also that, for a $n\times n$ symmetric matrix polynomial, we {can choose} $x_i(\lambda) = y_i(\lambda)$ and thus $\eta_i = \varepsilon_i,$ for $i = 1, \dots , n-r$.

The {\it complete eigenstructure} of a matrix polynomial $P(\lambda)$ is the collection of all the elementary divisors and the left and right minimal indices of $P(\lambda)$.

Now, we are in the position of introducing {\em orbits and bundles} of matrix polynomials.
Since in this paper we are mainly interested in symmetric matrix polynomials, we introduce these notions only for this type of polynomials.
Given a symmetric matrix polynomial $P(\lambda)$, the set of symmetric matrix polynomials with the same size, grade, and complete eigenstructure as $P(\lambda)$ is called the {symmetric} \emph{orbit} of $P(\lambda)$, denoted $\orb^{ s}(P)$. Similarly, the set of symmetric matrix polynomials with the same size, grade, and complete eigenstructure as $P(\lambda)$, except that the values of the distinct eigenvalues are unspecified as long as they remain distinct, is called the {symmetric} \emph{bundle} of $P(\lambda)$ and denoted by $\bun^{ s}(P)$. Note that $\bun^{ s}(P)$ is an infinite union (over all the possible values of the distinct eigenvalues) of the orbits of the polynomials whose complete eigenstructures differ from the one of $P(\la)$ only in the values of the distinct eigenvalues. Analogous definitions for general and skew-symmetric matrix polynomials are given in \cite{Dmyt15,dmydop-laa-2017,dmydop-laa-2018,DJKV19}. {It is important to emphasize that, unlike what happens for matrix pencils, two symmetric matrix polynomials with the same size, grade, and complete eigenstructure are not necessarily congruent. Consider, for instance, the symmetric polynomials of grade $6$
$$
P_1(\la)=\left[\begin{array}{ccc}1&\la^2&\la^3\\\la^2&\la^4&\la^5\\\la^3&\la^5&\la^6\end{array}\right]\qquad\mbox{and}\qquad
P_2(\la)=\left[\begin{array}{ccc}1&\la&\la^3\\\la&\la^2&\la^4\\\la^3&\la^4&\la^6\end{array}\right].
$$
$P_1(\la)$ and $P_2(\la)$ have the same complete eigenstructure, which consists only of left and right minimal indices both equal to $\{1,2\}$. However, $P_1(\la)$ and $P_2(\la)$ cannot be congruent, since $P_1(\la)$ has a zero coefficient in $\la$ while $P_2(\la)$ has a non-zero one. Therefore, the congruence orbit, as defined in \eqref{equorbit} for symmetric pencils, is not the appropriate notion to deal with symmetric matrix polynomials, since this orbit does not contain all symmetric matrix polynomials with the same size, grade, and complete eigenstructure. This is the reason for introducing the symmetric orbit for symmetric matrix polynomials in a different way than for pencils.}

Though the results presented in the remaining part of this section and in the following section are valid for an arbitrary infinite field $\bF$, we state them in the complex field $\mathbb{C}$ to be consistent with the rest of the paper.

We will often consider minimal bases of rational subspaces of $\mathbb{C}^n$  arranged as rows of polynomial matrices and for brevity we call such matrices also ``minimal bases'', with a clear abuse of nomenclature. An important related concept is that of {\em dual minimal bases} revisited in the following definition.

\begin{definition} {\rm \cite[Def. 2.10]{dtdvd-laa-zigzag-2016}} \label{def.dual-m-b}
{Two p}olynomial matrices $M(\lambda)\in \bC[\lambda]^{m\times n}$
  and $N(\lambda)\in \bC[\lambda]^{k\times n}$ with full row ranks
  are said to be {\em dual minimal bases}
  if they are minimal bases satisfying $m+k=n$ and $M(\lambda) \, N(\lambda)^\top=0$.
\end{definition}

As explained in \cite[p. 468]{dtdvd-laa-zigzag-2016}, the dual minimal bases described in Definition \ref{def.dual-m-b} satisfy the property stated in the next proposition, where the expression ``row degrees'' of a polynomial matrix means the degrees of the rows of such polynomial matrix.

\begin{proposition} \label{prop.dualyminind} Let $M(\lambda)\in \bC[\lambda]^{m\times n}$ and $N(\lambda)\in \bC[\lambda]^{k\times n}$ be dual minimal bases. Then, the right minimal indices of $M(\la)$ are the row degrees of $N(\la)$ and the left minimal indices of $N(\la)^\top$ are the row degrees of $M(\la)$.
\end{proposition}

A result that is fundamental in Section \ref{sec.existence} is the next theorem.

\begin{theorem} {\rm \cite[Thm. 6.1]{dtdvd-laa-zigzag-2016}} \label{thm.exist-dual-min-bases}
Let  $(\eta_1, \eta_2, \dotsc, \eta_m)$  and $(\veps_1, \veps_2, \dotsc, \veps_k)$
be any two lists of nonnegative integers such that
\begin{equation} \label{eqn.2keyequality}
  \sum_{i=1}^m \eta_i \,=\, \sum_{j=1}^k \varepsilon_j \, .
\end{equation}
Then, there exist two matrix polynomials $M(\la)\in\bC[\la]^{m\times (m+k)}$
and $N(\la)\in\bC[\la]^{k\times (m+k)}$ that are dual minimal bases,
and whose row degrees are $(\eta_1, \eta_2, \dotsc, \eta_m)$
and $(\veps_1, \veps_2, \dotsc, \veps_k)$, respectively.
\end{theorem}

An important property of any minimal basis $M(\la) \in\bC[\la]^{m\times (m+k)}$, with $k >0$, is that it can be completed to a unimodular matrix (see, for instance, \cite[Lemma 2.16 (b)]{DeDV15}, though this property has been reproved many times in the literature). This means that there exists a polynomial matrix $Z(\la) \in\bC[\la]^{k\times (m+k)}$ such that
\begin{equation} \label{eq.completion}
\begin{bmatrix}
  M(\la) \\
  Z(\la)
\end{bmatrix} \in \bC [\la] ^{(m+k) \times (m+k)} \quad \mbox{is unimodular.}
\end{equation}

The next trivial result is valid for matrices with entries in any field $\bF$ (in particular, for matrices with entries in the field of complex rational functions $\bC (\la)$ and, thus, for complex matrix polynomials) and will be also used in Section \ref{sec.existence}.

\begin{lemma} \label{lemm.trivial} Let $A \in \bF^{m \times r}$, $B \in \bF^{r \times r}$, and $C \in \bF^{r \times n}$ be three matrices with rank $r$. Then, the product $ABC \in \bF^{m \times n}$ has also rank $r$, the left null-space of $ABC$ is equal to the left null-space of $A$, and the right null-space of $ABC$ is equal to the right null-space of $C$.
\end{lemma}

Finally, we introduce a bit more of notation. The vector space, over the field $\bC$, of $n \times n$ matrix polynomials of grade $d$ with complex coefficients is denoted by $\POL_{d,n\times n}$, and its subspace of $n \times n$ symmetric matrix polynomials of grade $d$ with complex coefficients by $\POL^s_{d,n\times n}$.

\section{Symmetric matrix polynomials with prescribed eigenstructure in the case of different elementary divisors of degree 1} \label{sec.existence}

This section is devoted to the following inverse problem: given a list of $s$ linear (scalar) polynomials with different roots and a list of nonnegative numbers, when is there a symmetric matrix polynomial with prescribed size $n\times n$, grade $d$, and rank $r$, and having the given lists as the list of finite elementary divisors and right minimal indices, respectively, together with $t\in\{0,1\}$ infinite linear elementary divisors? It is known that the Index Sum Theorem \cite[Thm. 6.5]{DeDM14} must be satisfied by such a symmetric matrix polynomial, namely, $s$ (sum of all degrees of the finite elementary divisors), $t$ (sum of all degrees of the infinite elementary divisors), and the sum of twice the nonnegative integers (namely, the sum of the left and right minimal indices) must add up to $dr$. Theorem \ref{thm.existence} solves the inverse problem by showing that this condition is also sufficient. We advance that Theorem \ref{thm.existence} is key for proving the main result in this paper, i.e., Theorem \ref{mainth}. The reason is that Theorem \ref{thm.existence} guarantees that there exist symmetric matrix polynomials realizing the generic eigenstructures identified in Theorem \ref{mainth}.

\begin{theorem} \label{thm.existence} Let $n,d,$ and $r \leq n$ be positive integers. Let
\begin{enumerate}
  \item $(\lambda-\mu_1), \ldots, (\lambda-\mu_s)$ be $s$ scalar {monic} polynomials of degree $1$, where $\mu_1, \ldots , \mu_s \in \mathbb{C}$ satisfy $\mu_i \ne \mu_j$ if $i\ne j$,
  \item $t \in \{0,1\}$, and
  \item $\varepsilon_1 , \ldots , \varepsilon_{n-r}$ be a list of nonnegative integers.
\end{enumerate}
Then, there exists an $n \times n$ symmetric matrix polynomial, $P(\lambda) \in \mathbb{C} [\lambda]^{n \times n}$, with rank $r$, {\bf grade} $d$, finite elementary divisors $(\lambda-\mu_1), \ldots, (\lambda-\mu_s)$, $t$ infinite elementary divisors of degree $1$, and right minimal indices equal to $\varepsilon_1 , \ldots , \varepsilon_{n-r}$ (and, as a consequence, also left minimal indices equal to $\varepsilon_1 , \ldots , \varepsilon_{n-r}$), if and only if the following condition holds
\begin{itemize}
  \item[\rm (a)] $\displaystyle s + t + 2 \, \sum_{i=1}^{n-r} \varepsilon_{i} = r \, d$.
\end{itemize}

Moreover, assuming that {\rm(a)} holds, then the {\bf degree} of $P(\lambda)$ is exactly equal to $d$ if and only if $t < r$.
\end{theorem}

\begin{proof}
Necessity of (a). If there exists $P(\la)$ with the properties {in} the statement, then (a) follows from the Index Sum Theorem \cite[Thm. 6.5]{DeDM14}.

Sufficiency of (a). We consider a M\"obius transformation
$$m_A : \bC \cup \{\infty\} \longrightarrow \bC \cup \{\infty\} \quad \mbox{defined as} \quad m_A (\la) := (a \la + b)/(c \la + d)$$
in terms of the nonsingular matrix
$A = \left[\begin{smallmatrix}
a & b \\
c & d
\end{smallmatrix}\right] \in \mathrm{GL} (2, \bC)$
and such that $\la_1 := m_A (\mu_1), \ldots, \la_s := m_A (\mu_s), \theta := m_A (\infty)$ {all belong to} $\bC$ {(namely, they are not $\infty$)}. Then, the proof proceeds by constructing a {\em symmetric} matrix polynomial $Q(\la) \in \POL^s_{d,n\times n}$ with rank $r$ and whose complete eigenstructure consists of the following finite elementary divisors
\begin{equation} \label{eq.Qeldiv}
(\la - \la_1), \ldots , (\la - \la_s), \underbrace{(\la - \theta)}_{t},
\end{equation}
the right minimal indices $\varepsilon_1 , \ldots , \varepsilon_{n-r}$, and the left minimal indices also equal to $\varepsilon_1 , \ldots , \varepsilon_{n-r}$. Once $Q(\la)$ is constructed, we consider the M\"obius transformation $M_A : \POL_{d,n\times n} \longrightarrow \POL_{d,n\times n}$ as defined in
\cite[Def. 3.4]{MMMM-laa-mobius-2015}. Observe that $P(\la) = M_A (Q) (\la) \in \POL_{d,n\times n}$ is the desired matrix polynomial, because $P(\la)$ is symmetric if $Q (\la)$ is by the definition of $M_A$, $\rank P = \rank Q$ by \cite[Prop. 3.29]{MMMM-laa-mobius-2015}, the elementary divisors of $P(\la)$ are $(\lambda-\mu_1), \ldots, (\lambda-\mu_s)$ and $t$ infinite elementary divisors of degree $1$ by \cite[Thm. 5.3]{MMMM-laa-mobius-2015}, and the minimal indices of $P(\la)$ and $Q(\la)$ are identical by \cite[Thm. 7.5]{MMMM-laa-mobius-2015}.

In the case $r=n$, it is very easy to construct such a  $Q(\la)$ because there are no minimal indices and condition (a) implies that the number of linear elementary divisors in the list \eqref{eq.Qeldiv} is $r\, d = n\, d$. Then, these elementary divisors can be arranged into $n$ groups with $d$ elementary divisors each, and from each of these $n$ groups a scalar polynomial of degree $d$ can be obtained as the product of the elementary divisors in the group. We obtain in this way $n$ scalar polynomials $q_{11} (\la), \ldots, q_{nn}(\la)$ of degree exactly $d$. Then $Q(\la) = \mbox{diag} (q_{11} (\la), \ldots, q_{nn}(\la))$ is the desired matrix polynomial, since it is symmetric, has degree exactly $d$, has rank exactly $r=n$, has no minimal indices, has no eigenvalues at $\infty$, because the matrix coefficient of degree $d$ of $Q(\la)$ is invertible, and has the elementary divisors in the list \eqref{eq.Qeldiv}, as a consequence of \cite[Vol. I, Thm. 5, p. 142]{Gant59} and the fact that the elementary divisors in \eqref{eq.Qeldiv} are all different from each other. Thus, in the rest of the proof we assume that $r<n$.

For $r<n$, $Q(\la)$ will be constructed as the product of three matrix polynomials, each of them of rank $r$, i.e.,
\begin{equation}\label{eq.factorformofQ}
Q(\la) = U(\la) S(\la) U(\la)^\top
\end{equation}
with $U(\la) \in \bC [\la]^{n\times r}$, and $S(\la) \in \bC [\la]^{r\times r}$ either diagonal or with a very simple symmetric structure. Thus, $Q(\la)$ is symmetric by construction.

Let us construct first the matrix $U(\la)$ in \eqref{eq.factorformofQ}. For this purpose, we consider the Euclidean division of ${\varepsilon:=}\sum_{i=1}^{n-r} \varepsilon_{i}$ by $r$
\begin{equation}\label{eq.eucldiv2}
 {\varepsilon}= r q_\veps + w_\veps, \quad \mbox{where $0\leq w_\veps < r$},
\end{equation}
and we define the list
\begin{equation}\label{eq.degreesU}
(\eta_1, \ldots , \eta_r) = (\underbrace{q_\veps, \ldots , q_\veps}_{r-w_\veps}, \underbrace{q_\veps + 1, \ldots , q_\veps +1}_{w_\veps}).
\end{equation}
Note that $\sum_{i=1}^{r} \eta_i = {\varepsilon}$
. Therefore, Theorem \ref{thm.exist-dual-min-bases} guarantees that there exist dual minimal bases $M(\la)\in\bC[\la]^{(n-r)\times n}$
and $N(\la)\in\bC[\la]^{r\times n}$ with row degrees $(\veps_1, \veps_2, \dotsc, \veps_{n-r})$ and $(\eta_1, \eta_2, \dotsc, \eta_r)$, respectively. We take $U(\la) = N(\la)^\top$. Observe that, according to Proposition \ref{prop.dualyminind}, the left minimal indices of $U(\la)$ are precisely $(\veps_1, \veps_2, \dotsc, \veps_{n-r})$, or, equivalently, the right minimal indices of $U(\la)^\top$ are $(\veps_1, \veps_2, \dotsc, \veps_{n-r})$.

Next we construct $S(\la)$. Observe that, from (a) and \eqref{eq.eucldiv2},
\begin{equation}\label{eq.inequality}
0 \leq s + t + 2 w_\veps = r (d - 2 \, q_\veps) ,
\end{equation}
and, thus, $d - 2 \, q_\veps \geq 0$. We need to consider three different cases.

Case 1. $d - 2 \, q_\veps = 0$. According to \eqref{eq.inequality}, this condition implies
$s = t = w_\veps = 0$. Thus, in this case, there are no elementary divisors in the list \eqref{eq.Qeldiv} and the column degrees of $U(\la)$ are all equal to $q_\veps$. The desired matrix polynomial is $Q(\la) = U(\la) U(\la)^\top$ (i.e., with $S(\la) = I_r$), which has
\begin{itemize}
\item grade $2 q_\veps = d$,
\item rank equal to $r$ by Lemma \ref{lemm.trivial} and because $U(\la)\in\bC[\la]^{n\times r}$ has rank $r$,
\item left and right minimal indices equal to $(\veps_1, \veps_2, \dotsc, \veps_{n-r})$ by Lemma \ref{lemm.trivial} and because the left minimal indices of $U(\la)$ are $(\veps_1, \veps_2, \dotsc, \veps_{n-r})$,
\item no finite elementary divisors because $U(\la)$ and $U(\la)^\top$ can be completed to $n\times n$ unimodular matrices $V(\la)$ and $V(\la)^\top$ as a consequence of \eqref{eq.completion} and, so, $Q(\la) = V(\la) \, \diag (I_r, 0) \, V(\la)^\top$,
\item no infinite elementary divisors because, otherwise, $dr =  r 2 q_\veps = 2\, {\varepsilon}
< 2\,{\varepsilon}
 + \mbox{``Sum of the degrees of the infinite elementary divisors''}$ and the Index Sum Theorem would be violated.
\end{itemize}

Case 2. $d - 2 \, q_\veps = 1$. In this case, \eqref{eq.inequality} implies that $s+t = r - 2w_\veps$ and $d = 2 \, q_\veps + 1$. In this situation, it is more convenient to see \eqref{eq.degreesU} as
\begin{equation}\label{eq.degreesU2}
(\eta_1, \ldots , \eta_r) = (\underbrace{q_\veps, \ldots , q_\veps}_{r-2w_\veps}, \underbrace{q_\veps, \ldots , q_\veps}_{w_\veps}, \underbrace{q_\veps + 1, \ldots , q_\veps +1}_{w_\veps}).
\end{equation}
The matrix $S(\la)$ has the following structure
\begin{equation}\label{eq.particularS}
S(\la) = \begin{bmatrix}
           S^{(1)}(\la) & 0 & 0 \\
           0 & 0 & I_{w_\veps} \\
           0 & I_{w_\veps} & 0
         \end{bmatrix},
\end{equation}
where $S^{(1)}(\la) = \diag (S_{11} (\la), \ldots , S_{r - 2w_\veps,r- 2w_\veps} (\la)) \in \bC [\la]^{(r - 2w_\veps) \times (r - 2w_\veps)}$ is a diagonal matrix with exactly one of the elementary divisors in the list \eqref{eq.Qeldiv} in each of its diagonal entries. The matrix $S(\la)$ satisfies obviously the following properties:
\begin{enumerate}
  \item The polynomials $S_{11} (\la), \ldots, S_{r-2w_\veps , r-2w_\veps} (\la)$ have all degree $1$.
  \item As a consequence of \cite[Vol. I, Thm. 5, p. 142]{Gant59}, the finite elementary divisors of $S(\la)$ are precisely the linear scalar polynomials in \eqref{eq.Qeldiv}.
  \item $S(\la)$ is obviously nonsingular and, so, $\rank S = r$.
\end{enumerate}
If the columns of the matrix $U(\la) \in \bC [\la]^{n \times r}$ with column degrees given by the list \eqref{eq.degreesU2} are partitioned as follows
\[
\begin{array}{c}
U(\la) = \begin{bmatrix}
           U^{(1)} (\la) & U^{(2)} (\la) & U^{(3)} (\la)
\end{bmatrix} \\[-0.4cm]
\phantom{---...} \begin{array}{ccc} \!
\underbrace{\phantom{.---.}}_{r-2w_\veps} & \, \underbrace{\phantom{---...}}_{w_\veps} & \; \underbrace{\phantom{.---.}}_{w_\veps}
\end{array}
\end{array},
\]
then from \eqref{eq.factorformofQ}
\[
Q(\la) = U^{(1)}(\la) S^{(1)}(\la) U^{(1)}(\la)^\top + U^{(2)}(\la) U^{(3)}(\la)^\top + U^{(3)}(\la) U^{(2)}(\la)^\top,
\]
where each of the three {summands} has grade $2 \, q_\veps + 1 = d$. Thus $Q(\la)$ has
\begin{itemize}
\item grade $d$,
\item rank equal to $r$ by Lemma \ref{lemm.trivial} and because $U(\la)\in\bC[\la]^{n\times r}, S(\la)\in\bC[\la]^{r\times r}$ have both rank $r$,
\item left and right minimal indices equal to $(\veps_1, \veps_2, \dotsc, \veps_{n-r})$ by Lemma \ref{lemm.trivial} and because the left minimal indices of $U(\la)$ are $(\veps_1, \veps_2, \dotsc, \veps_{n-r})$,
\item the same finite elementary divisors as $S(\la)$ because $U(\la)$ and $U(\la)^\top$ can be completed to $n\times n$ unimodular matrices $V(\la)$ and $V(\la)^\top$ as a consequence of \eqref{eq.completion} and, so, $Q(\la) = V(\la) \, \diag (S(\la), 0) \, V(\la)^\top$,
\item no infinite elementary divisors because, otherwise, $dr =  s+t+2\, {\varepsilon}
 < s+t + 2\, {\varepsilon}
 + \mbox{``Sum of the degrees of the infinite elementary divisors''}$ and the Index Sum Theorem would be violated.
\end{itemize}

Case 3. $d - 2 \, q_\veps \geq 2$. Note that in this case the number of elementary divisors in the list \eqref{eq.Qeldiv} is, as a consequence of (a) and \eqref{eq.eucldiv2},
\begin{align*}
s + t & = rd - 2\, {\varepsilon}
 \\
      & = (r - w_\veps) (d- 2 \, q_\veps ) + w_\veps (d- 2 \, q_\veps -2).
\end{align*}
This means that we can arrange the $s+t$ elementary divisors in \eqref{eq.Qeldiv} into $r - w_\veps$ groups with $d- 2 \, q_\veps$ elementary divisors each, {together with} $w_\veps$ groups with $d- 2 \, q_\veps - 2$ elementary divisors each. Once this grouping is done we construct $S(\la)$ as a diagonal matrix that has in each of its first $r - w_\veps$ diagonal entries the product of the elementary divisors of each of the groups of $d- 2 \, q_\veps$ elementary divisors and in each of its last $w_\veps$ diagonal entries the product of the elementary divisors of each of the groups of $d- 2 \, q_\veps - 2$ elementary divisors. The matrix $S(\la) = \diag (S_{11} (\la), \ldots , S_{rr} (\la))$ constructed in this way satisfies the following properties:
\begin{enumerate}
  \item The polynomials $S_{11} (\la), \ldots, S_{r-w_\veps , r-w_\veps} (\la)$ have all degree $d- 2 \, q_\veps$ and the polynomials $S_{r-w_\veps+1,r-w_\veps+1} (\la), \ldots, S_{rr} (\la)$ have all degree $d- 2 \, q_\veps - 2$.
  \item As a consequence of \cite[Vol. I, Thm. 5, p. 142]{Gant59} and the fact that the linear finite elementary divisors in the list \eqref{eq.Qeldiv} are all different from each other, the finite elementary divisors of $S(\la)$ are precisely the linear scalar polynomials in \eqref{eq.Qeldiv}.
  \item $S(\la)$ is obviously nonsingular and, so, $\rank S = r$.
\end{enumerate}
Then the matrix $Q(\la)$ in \eqref{eq.factorformofQ} can be written in this case as
\[
Q(\la) = {S_{11} (\la)}\,\mbox{col}_1 (U) \,  \mbox{col}_1 (U)^\top + \cdots +{ S_{rr} (\la) \,}
\mbox{col}_r (U) \, \mbox{col}_r (U)^\top,
\]
where each of the terms $ S_{ii} (\la)\, \mbox{col}_i (U) \,\mbox{col}_i (U)^\top$ has degree exactly $d$ as a consequence of the discussion above and \eqref{eq.degreesU}. Thus $Q(\la)$ has
\begin{itemize}
\item grade $d$,
\item rank equal to $r$ by Lemma \ref{lemm.trivial} and because $U(\la)\in\bC[\la]^{n\times r}, S(\la)\in\bC[\la]^{r\times r}$ have both rank $r$,
\item left and right minimal indices equal to $(\veps_1, \veps_2, \dotsc, \veps_{n-r})$ by Lemma \ref{lemm.trivial} and because the left minimal indices of $U(\la)$ are $(\veps_1, \veps_2, \dotsc, \veps_{n-r})$,
\item the same finite elementary divisors as $S(\la)$, because $U(\la)$ and $U(\la)^\top$ can be completed to $n\times n$ unimodular matrices $V(\la)$ and $V(\la)^\top$, as a consequence of \eqref{eq.completion}, and, so, $Q(\la) = V(\la) \, \diag (S(\la), 0) \, V(\la)^\top$,
\item no infinite elementary divisors because, otherwise, $dr =  s+t+2\,{\varepsilon}
 < s+t + 2\,{\varepsilon}
+ \mbox{``Sum of the degrees of the infinite elementary divisors''}$ and the Index Sum Theorem would be violated.
\end{itemize}
Finally, we prove that, if (a) holds, then the degree of $P(\lambda)$ is exactly equal to $d$ if and only if $t < r$. We have proved the existence of an $n\times n$ symmetric matrix polynomial of rank $r$ and grade $d$
$$
P(\la) = P_d \la^d + \cdots + P_1 \la + P_0
$$
with the prescribed eigenstructure. Note that
$$
(\mathrm{rev} P) (\la) = P_0 \la^d + \cdots + P_{d-1} \la + P_d.
$$
Let
$$
(\mathrm{rev} P) (\la) = {V}(\la) \diag(p_1 (\la), \ldots , p_r (\la), 0_{(n-r) \times (n-r)}) {W}(\la)
$$
be the Smith form of $(\mathrm{rev} P) (\la)$ with ${V}(\la)$ and ${W}(\la)$ unimodular matrices. Then, $P_d = (\mathrm{rev} P) (0) = {V}(0) \diag(p_1 (0), \ldots , p_r (0), 0_{(n-r) \times (n-r)}) {W}(0) \ne 0$ if and only if $p_i (0) \ne 0$ for at least one $i$, i.e., if and only if $t < r$.
\end{proof}


\section{Symmetric linearization of symmetric matrix polynomials}\label{sect.linz}

A matrix pencil ${\cal F}_P{(\lambda)}$ is called a {\it linearization} of a matrix polynomial $P(\lambda)$ if they have the same finite elementary divisors, the same number of left minimal indices, and the same number of right minimal indices \cite{DeDM14}. If, in addition, $P(\la)$ is considered as a polynomial of grade $d$ (recall that, by definition, matrix pencils have grade $1$) and $\rev {\cal F}_P{(\lambda)}$ is a linearization of $\rev P(\lambda)$, then ${\cal F}_P{(\lambda)}$ is called a {\it strong linearization} of $P(\lambda)$ and, then, ${\cal F}_P{(\lambda)}$ and $P(\lambda)$ have also the same infinite elementary divisors. Note that the notion of strong linearization depends on the grade $d$ chosen for the polynomial $P(\la)$ through the definition of $\rev P(\lambda)$. In this sense, we emphasize that the linearizations we introduce in this section will be strong linearizations for any matrix polynomial when it is considered as a polynomial with any odd grade $d$ larger than or equal to its degree.

From now on we restrict ourselves to symmetric matrix polynomials of odd grades. The reason is that there is no symmetric linearization-template (i.e., a symmetric companion form in the language of \cite[Sects. 5 and 7]{DeDM14}) for symmetric matrix polynomials of even grades \cite{DeDM14,MMMM13}. As in the rest of the paper, we also restrict ourselves to considering polynomials over the field of complex numbers $\mathbb C$, though some of the results in this section remain valid over arbitrary fields.

We recall the following pencil-template, ${\cal F}_P(\lambda)$, which is a symmetric strong linearization {for any} symmetric $n\times n$ matrix polynomial $P(\lambda) = \la^d A_d + \cdots + \la A_1 + A_0$ of odd grade $d$, see e.g., \cite{AnVo04,MMMM10,MMMM13}. It is a $d\times d$ block-partitioned pencil with blocks of size $n\times n$. Therefore, ${\cal F}_P(j,k)$ below denotes the $(j,k)$ block in this partition, for $j,k=1,\dots,d$ (we drop the dependence on $\la$ for simplicity).
\begin{align*}
{\cal F}_P(i,i)&=\begin{cases}
\lambda A_{d-i+1} + A_{d-i} & \text{if } i \text{ is odd,}\\
0 & \text{if } i \text{ is even,}\\
\end{cases}\\
{\cal F}_P(i,i+1)&=\begin{cases}
 I_n & \text{if } i \text{ is odd,}\\
 \lambda I_n & \text{if } i \text{ is even,}\\
\end{cases} \quad
{\cal F}_P(i+1,i)=\begin{cases}
I_n & \text{if } i \text{ is odd,}\\
\lambda I_n & \text{if } i \text{ is even.}\\
\end{cases}
\end{align*}
The blocks of ${\cal F}_P{(\lambda)}$ in positions which are not specified above are zero. We rewrite this strong linearization template in  a matrix form:
\hide{
\begin{equation}
\label{linform}
{\cal F}_P{(\lambda)}=
\begin{bmatrix}
\lambda A_d + A_{d-1}& -I &&&&&\\
I &0&-\lambda I&&&&\\
&\lambda I &\ddots&\ddots&&&\\
&&\ddots&0&-\lambda I&&\\
&&&\lambda I&\lambda A_{3}+A_{2}&-I&\\
&&&&I&0&-\lambda I\\
&&&&&\lambda I&\lambda A_1+ A_0\\
\end{bmatrix}
\end{equation}
or }
\begin{equation}
\label{linform}
{\cal F}_P{(\lambda)}=
\lambda
\begin{bmatrix}
A_d&&&&&\\
&\ddots&\ddots&&&\\
&\ddots&0&I&&\\
&&I&A_{3}&&\\
&&&&0&I\\
&&&&I&A_1\\
\end{bmatrix} + \begin{bmatrix}
A_{d-1}&I&&&&\\
I&0&\ddots&&&\\
&\ddots&\ddots&&&\\
&&&A_{2}&I&\\
&&&I&0&\\
&&&&&A_0\\
\end{bmatrix}.
\end{equation}
For a symmetric matrix polynomial {$P(\la)$}, {the} strong linearization \eqref{linform} preserves {both} finite and infinite elementary divisors of $P(\lambda)$, but does not preserve the left and right minimal indices of $P(\lambda)$. Nevertheless, the results {from} \cite{DeDM10, DeDM12, Dmyt15} {allow us to derive} the relation between the minimal indices of a symmetric matrix polynomial $P(\lambda)$ and its linearization \eqref{linform}{, which is} presented in the following theorem.

\begin{theorem} \label{skewlin}
Let $P(\lambda)$ be a symmetric $n\times n$ matrix polynomial of odd grade $d \ge 3$, and let ${\cal F}_P{(\lambda)}$ be its strong linearization \eqref{linform}. If $0 \le  \varepsilon_1 \le \varepsilon_2 \le~\dots \le \varepsilon_t$ are the right ({and also }left) minimal indices of $P(\lambda)$, then
$$0 \le  \varepsilon_1 + \frac{1}{2}(d-1) \le \varepsilon_2 + \frac{1}{2}(d-1) \le  \cdots \le \varepsilon_t + \frac{1}{2}(d-1)$$
are the right ({and also} left) minimal indices of ${\cal F}_P{(\lambda)}$.
\end{theorem}
\begin{proof}
See the proof of Theorem 4.1 in \cite{Dmyt15}, {and} also \cite{DeDM10, DeDM12} for such results for general matrix polynomials.
\end{proof}

The strong linearization ${\cal F}_P{(\lambda)}$ in \eqref{linform} is crucial for obtaining the results in Section~\ref{sec.main}. {To this end,} we define the {\it generalized Sylvester space} consisting of the linearizations ${\cal F}_P{(\lambda)}$ of all the $n \times n$ symmetric matrix polynomials of odd grade $d$:
\begin{equation}\label{gsyl}
\begin{aligned}
\GSYL^{s}_{d,n\times n}{:}= \{ {\cal F}_P{(\lambda)}  \ : P(\lambda) &\text{ {is an} } n \times n \text{ symmetric } \\
&\text{matrix polynomial of odd grade } d \}.
\end{aligned}
\end{equation}
If there is no risk of confusion, we will write $\GSYL$ instead of $\GSYL^{s}_{d, n\times n}$, as well as $\POL$ instead of $\POL_{d, n\times n}^{s}$, especially  in  explanations and proofs.


As we did with matrix pencils in Section \ref{sect.prempencils}, we define the distance between $P(\lambda)=\sum_{i=0}^d \lambda^i A_i$ and $\widetilde P(\lambda)=\sum_{i=0}^d \lambda^i \widetilde A_i$ on $\POL_{d, n\times n}^{s}$ as $d(P,\widetilde P) = \sqrt{ \sum_{i=0}^d || A_i - \widetilde A_i ||_F^2}$, making $\POL_{d, n\times n}^{s}$ a metric space  with the induced Euclidean topology and allowing us to consider closures of subsets of $\POL_{d, n\times n}^{s}$. For convenience,  define the Frobenius norm of the matrix polynomial $P(\la)$ as $||P(\lambda)||_F = \sqrt{\sum_{i=0}^d || A_i ||_F^2 }$.

Similarly, given the linearizations ${\cal F}_P{(\lambda)} = \lambda A {+}B$ and ${\cal F}_{\widetilde P}{(\lambda)} = \lambda \widetilde A {+} \widetilde B$, the function $d({\cal F}_P,{\cal F}_{\widetilde P}):= \sqrt {||A-\widetilde A||_F^2 + ||B-\widetilde B||_F^2}$
is a distance on $\GSYL${, which} makes $\GSYL$ a metric space. Since $d({\cal F}_P, {\cal F}_{\widetilde P}) = d(P,\widetilde P)$, there is a bijective isometry (and, therefore, a homeomorphism):
\begin{equation}\label{homeo}
f: \POL^{s}_{d, n\times n} \rightarrow \GSYL^{s}_{d, n\times n} \quad \text{such that} \quad f(P) = {\cal F}_{P}.
\end{equation}
Next we define the orbit and bundle of the symmetric linearizations \eqref{linform} of a{n $n\times n$} symmetric matrix polynomial $P$ {of grade $d$}:
\begin{equation}\label{linorb}
\begin{aligned}
\orb^{\rm{syl}}({\cal F}_{P}) &{:}= \orb^c({\cal F}_{P})  \cap  \GSYL^{s}_{d, n\times n} \\
&= \{(S^{\top}{\cal F}_P{(\lambda)} S) \in \GSYL^{s}_{d, n\times n} \ : \ S \in  GL_{nd}(\mathbb C) \} \quad \text{ and } \\
\bun^{\rm{syl}}({\cal F}_{P}) &{:}= \bun^c({\cal F}_{P})  \cap  \GSYL^{s}_{d, n\times n}.
\end{aligned}
\end{equation}
Notably, all the elements of $\orb^{\rm{syl}}({\cal F}_{P})$ and $\bun^{\rm{syl}}({\cal F}_{P})$ have the block structure of the elements of $\GSYL^{s}_{d, n\times n}$. Thus, in particular, $\bun^{ s}(P) = f^{-1}(\bun^{\rm{syl}}({\cal F}_{P}))$, {with $f$ being the homeomorphism (isometry) in \eqref{homeo},} as a consequence of the properties of strong linearizations and Theorem \ref{skewlin}, and $\overline{\bun^{ s}}(P) = f^{-1}\left(\overline{\bun^{\rm{syl}}}({\cal F}_{P})\right)$, as a consequence of $f$ being a homeomorphism. Moreover, for any $n \times n$ symmetric matrix polynomials $P, Q$ of odd grade $d$,
{\begin{equation}\label{equivalence}
\overline{\bun^s}(P) \supseteq \overline{\bun^s}(Q)\quad\mbox{ if and only if}\quad \overline{\bun^{\rm{syl}}}({\cal F}_{P}) \supseteq \overline{\bun^{\rm{syl}}}({\cal F}_{Q}),
\end{equation}}where it is essential to note that the closures are taken{, respectively,} in the metric spaces $\POL$ and $\GSYL$ {defined above}.

The following theorem states that, for any $n\times n$ symmetric matrix polynomial $P(\lambda)$ of odd grade $d$, a sufficiently small arbitrary symmetric perturbation of ${\cal F}_P{(\lambda)}$ produces another pencil that, although, in general, is not in $\GSYL^{s}_{d, n\times n}$, is congruent to a pencil in $\GSYL^{s}_{d, n\times n}$ that is very close to ${\cal F}_P{(\lambda)}$. See \cite{Dmyt15, DLPVD15, DPVD16, JoKV13, VaDe83} for this type of results for linearizations of general, skew-symmetric, full rank matrix polynomials, and other structured classes of matrix polynomials and linearizations.

\begin{theorem}\label{deform}
Let $P(\lambda)$ be an $n \times n$ symmetric  matrix polynomial with odd grade $d$ and let ${\cal F}_P{(\lambda)}$ be its symmetric linearization~\eqref{linform}. If ${\mathcal E} (\la)$ is any arbitrarily small (entrywise for each coefficient) $nd \times nd$ symmetric pencil, then there exist a nonsingular matrix $C$ and an arbitrarily small (entrywise for each coefficient) $n \times n$ symmetric matrix polynomial $S(\lambda)$ with grade $d$ such that
\begin{equation*}
\label{eq19}
C^\top ({\cal F}_P{(\lambda)} +{\mathcal E} (\la)) C = {\cal F}_{P+S}{(\lambda)} \,.
\end{equation*}
\end{theorem}
\begin{proof}
The linearization ${\cal F}_P(\lambda)$ defined in \eqref{linform} differs from the skew-symmetric linearization of skew-symmetric polynomials considered in \cite{Dmyt15} only in the signs of some identity blocks and the fact that the matrix coefficients of the polynomials are symmetric.
Since the signs of the identity blocks do not affect the reduction process in the proof of Theorem 5.1 in \cite{Dmyt15}, we can repeat this process for the symmetric linearization ${\cal F}_P(\lambda)$ \eqref{linform} (note that the matrix coefficients are not used in the reduction process at all). Thus Theorem 5.1 in \cite{Dmyt15} remains true for the symmetric linearization ${\cal F}_P(\lambda)$ \eqref{linform} of the symmetric polynomials. Therefore (see also Theorem 7.1 in \cite{Dmyt15})  each sufficiently small symmetric perturbation of the linearization of a symmetric $n \times n$ matrix polynomial ${\cal F}_P(\lambda) +{\mathcal E} (\la)$ can be smoothly reduced by congruence to another one in which only the matrix coefficients are perturbed.

Note also that this result can be seen as a corollary of Theorem 6.13 in \cite{DPVD16}.
\end{proof}

\section{Generic symmetric matrix polynomials with bounded rank and fixed odd grade}
\label{sec.main}

In this section we describe the generic complete eigenstructures of $n\times n$  complex symmetric matrix polynomials with bounded rank $r$ and odd grade $d$. {The idea to obtain these generic eigenstructures is the following: Theorem \ref{th:improvedDeDo} displays the generic eigenstructures of symmetric pencils with bounded rank. We want to use this information to get the generic eigenstructures of symmetric matrix polynomials with bounded rank and fixed odd grade using that any such polynomial $P(\la)$ has a symmetric strong linearization ${\cal F}_P(\la)$ as in \eqref{linform}. Theorem \ref{deform} allows us to perform this transition. According to this result, any sufficiently small symmetric perturbation of such symmetric linearization of any symmetric polynomial is congruent to the linearization \eqref{linform} of a small symmetric perturbation of the polynomial. Roughly speaking, this means that small perturbations in the polynomial can be identified with small perturbations in the linearization and vice versa (up to congruence, which is enough from the point of view of the eigenstructure). However, it is essential to take into account that not every generic eigenstructure in Theorem \ref{th:improvedDeDo} corresponds to the eigenstructure of the linearization \eqref{linform} of a symmetric matrix polynomial of grade $d$. Lemma \ref{lem.eiglinz} identifies which ones are compatible with such polynomials, and which ones are not. We warn the reader that the dependence on $\la$ of matrix pencils and polynomials is often omitted for simplicity in the proofs in this section.

\begin{lemma}\label{lem.eiglinz}
Let $n,r$, and $d$ be integers such that $n\geq2$, $d\geq1$ is odd, and $1\leq r\leq n-1$. Set  $r_1:=n(d-1)+r$ and let ${\cal K}_{a_1} (\lambda)$, for $a_1=0,1,\ldots,\lfloor\frac{r_1}{2}\rfloor\,$, be anyone of the symmetric $nd \times nd$ pencils with rank $r_1$ in Theorem {\rm\ref{th:improvedDeDo}}, i.e., taking any possible choice of distinct eigenvalues. Define $\lfloor \frac{rd}{2} \rfloor + 1$ complete eigenstructures ${\mathbf K}_a$ of matrix polynomials with $rd-2a$ linear elementary divisors corresponding to arbitrary (finite) distinct eigenvalues, with $s$ left minimal indices equal to $\alpha +1$, $n-r-s$ another left minimal indices equal to $\alpha$, and with the right minimal indices equal to the left minimal indices, as follows:
\begin{equation}
\label{kcilist}
{\mathbf K}_a: \bigg\{\overbrace{\underbrace{\alpha+1, \dots , \alpha+1}_{s},\underbrace{\alpha, \dots , \alpha}_{n-r-s}}^{\text{left minimal indices}}, \overbrace{\underbrace{\alpha+1, \dots , \alpha+1}_{s}, \underbrace{\alpha, \dots ,\alpha}_{n-r-s}}^{\text{right minimal indices}}, (\lambda - \mu_1), \dots, (\lambda - \mu_{rd-2a}) \bigg\}
\end{equation}
for $a = 0,1,\dots,\lfloor \frac{rd}{2} \rfloor$, where $\alpha = \lfloor a / (n-r) \rfloor$ and $s = a  \mod (n-r)$. Then
\begin{itemize}
\item[{\rm(i)}] For each $a=0,1,\dots,\lfloor\frac{rd}{2}\rfloor$ and for any choice of distinct eigenvalues $\mu_1, \ldots , \mu_{rd-2a}$, there is an $n \times n$ complex symmetric matrix polynomial $K_a (\la)$ of degree exactly $d$ and rank exactly $r$ with complete eigenstructure ${\mathbf K}_a$.
\item[{\rm (ii)}] If $a_1<\frac{1}{2}(n-r)(d-1)$, then ${\cal K}_{a_1}(\la)$ is not congruent to the linearization \eqref{linform} of any $n\times n$ symmetric matrix polynomial of grade $d$.
\item[{\rm (iii)}] If $a_1\geq \frac{1}{2}(n-r)(d-1)$, then $\bun^c ({\cal K}_{a_1}) = \bun^c ({\cal F}_{K_a})$, with $a=a_1-\frac{1}{2}(n-r)(d-1)$.
\end{itemize}
\end{lemma}
\begin{proof}
(i) Summing up the number of elementary divisors and all the minimal indices for each ${ {\mathbf K}_a}$ in \eqref{kcilist} we have
\begin{align*}
&rd-2a + 2\left(\sum_1^s (\alpha +1) + \sum_1^{n-r-s} \alpha \right)
=rd-2a + 2\left( \sum_1^{n-r} \alpha + s \right)  \\
&= rd-2a + 2\left(  (n-r) \lfloor a / (n-r) \rfloor + s \right) = rd-2a +2a = rd.
\end{align*}
By Theorem \ref{thm.existence}, for each $a$ and for any list $\mu_1, \ldots , \mu_{rd - 2a}$ of distinct finite eigenvalues, there exists an $n\times n$ complex symmetric matrix polynomial $K_a$ of grade $d$ and rank exactly $r$ that has the complete eigenstructure~${ {\mathbf K}_a}$ {in} \eqref{kcilist}. Note that, since ${\mathbf K}_a$ does not contain infinite eigenvalues ($t=0$), the degree of $K_a$ is exactly $d$.

(ii) Set $n_1=nd$ and note, first, that $n-r=n_1-r_1$ (this identity will be used several times throughout the rest of the proof). Now, for $a_1=0,1,\dots,\lfloor\frac{r_1}{2}\rfloor$, the generic canonical forms of the symmetric matrix pencils of rank at most $r_1=n(d-1)+r$ and size $n_1 \times n_1$ given in Theorem \ref{th:improvedDeDo} are:
\begin{equation} \label{pg}
{\cal K}_{a_1}=\diag \left(\underbrace{{\cal M}_{\alpha_1+1}, \dots , {\cal M}_{\alpha_1+1}}_{s_1},\underbrace{{\cal M}_{\alpha_1}, \dots , {\cal M}_{\alpha_1}}_{n_1-r_1-s_1},
{\cal J}_{1}(\tilde{\mu}_1), \dots , {\cal J}_{1}(\tilde{\mu}_{r_1-2a_1}) \right),
\end{equation}
where $\alpha_1=\lfloor\frac{a_1}{n_1-r_1}\rfloor=\lfloor\frac{a_1}{n-r}\rfloor$ and $s_1 = a_1 \mod (n_1 - r_1)$. Therefore, if $a_1 <  \frac{1}{2}(n-r)(d-1)$, then the ${\cal M}_{\alpha_1}$ blocks correspond to minimal indices smaller than $(d-1)/2$. As a consequence of Theorem \ref{skewlin}, they cannot correspond to the linearization {\eqref{linform}} of any $n\times n$ symmetric matrix polynomial of grade $d$.

(iii) For each matrix polynomial $K_a$ in part (i) the $nd \times nd$ symmetric matrix pencil ${\cal F}_{K_a}$ has rank $n(d-1)+r$ and, by Theorem \ref{skewlin}, the Kronecker-type symmetric canonical form in Theorem \ref{kron} of ${\cal F}_{K_a}$ is
\begin{multline} \label{cka}
\diag \left(\underbrace{{\cal M}_{\alpha+\eta+1}, \dots , {\cal M}_{\alpha+\eta+1}}_{s},\underbrace{{\cal M}_{\alpha+\eta}, \dots , {\cal M}_{\alpha+\eta}}_{n-r-s},
{\cal J}_{1}(\mu_1), \dots , {\cal J}_{1}(\mu_{rd-2a}) \right),
\end{multline}
where $\eta:=\frac{1}{2}(d-1)$. We are going to show that, if $a_1\geq \frac{1}{2}(n-r)(d-1)$, then the canonical form ${\cal K}_{a_1}$ in \eqref{pg} coincides, up to the values of the eigenvalues, with the one of ${\cal F}_{K_a}$ in \eqref{cka}, for $a=a_1-\frac{1}{2}(n-r)(d-1)$, as claimed in the statement.

By the choice of $a$ and $a_1$, the number of eigenvalues in both \eqref{pg} and \eqref{cka} coincide, namely $r_1-2a_1=n(d-1)+r-2(a+\frac{1}{2}(n-r)(d-1))=rd-2a$. It remains to show that the numbers and the indices of the ${\cal M}$ blocks in \eqref{pg} and \eqref{cka} coincide as well, i.e. $\alpha  + \eta = \alpha_1$, $s=s_1$, and $n-r-s = n_1 - r_1 - s_1$.

For the indices of the ${\cal M}$ blocks we have
\begin{equation} \label{alphasize1}
\begin{array}{cll}
\displaystyle\alpha  +\eta& \displaystyle=\left\lfloor \frac{a}{n-r} \right\rfloor +  \frac{d - 1}{2} =  \left\lfloor  \frac{\frac{1}{2}(n-r)(d - 1) + a}{n-r} \right\rfloor  \\
& = \displaystyle\left\lfloor  \frac{a_1}{n_1-r_1} \right\rfloor  = \alpha_1,
\end{array}
\end{equation}
as claimed.

For the number of ${\cal M}$ blocks 
we have
\begin{align*}
s &= a \mod (n-r) = \left(\frac{1}{2}(n-r)(d - 1) + a\right) \mod (n-r) \\
&=  a_1 \mod (n_1 - r_1) =s_1.
\end{align*}
Therefore, $n-r-s=n_1-r_1-s_1$ as well, as wanted.
\end{proof}}

As in the case {of} general and skew-symmetric matrix polynomials \cite{dmydop-laa-2017,dmydop-laa-2018}, the following lemma {reveals} a key relation between $\overline{\bun^{{\rm syl}}}({\cal F}_{P})$, where the closure is taken in $\GSYL^{s}_{d,n\times n}$, and $\overline{\bun^c}({\cal F}_{P})$, where the closure is taken in $\PEN^{s}_{n_1 \times n_1}$, with $n_1 = n d$.

\begin{lemma}\label{cl}
Let $P(\la)$ be an $n\times n$ symmetric matrix polynomial with odd grade $d$ and let ${\cal F}_{P} (\la)$ be its linearization \eqref{linform}. Then $\overline{\bun^{{\rm syl}}}({\cal F}_{P})= \overline{\bun^c}({\cal F}_{P}) \cap \GSYL^{s}_{d,n\times n}$.
\end{lemma}
\begin{proof}
The proof of this lemma follows the proof of Lemma 5.1 in \cite{dmydop-laa-2018} but now it deals with the bundles and not with the orbits.
By definition, $\bun^{\rm{syl}}({\cal F}_{P})= \bun^c({\cal F}_{P}) \cap \GSYL$, and thus $\overline{\bun^{\rm{syl}}}({\cal F}_{P})= \overline{\bun^c({\cal F}_{P}) \cap \GSYL}$  (the closure here is taken in the space $\GSYL$). For any ${\cal F}_{Q} \in \overline{\bun^c}({\cal F}_{P})\cap \GSYL$ there exists an arbitrarily small  (entrywise) symmetric pencil ${\mathcal E}$ such that $({\cal F}_{Q} +{\mathcal E}) \in \bun^c({\cal F}_{P})$. Therefore by Theorem \ref{deform},
there exists an arbitrarily small  (entrywise) symmetric matrix polynomial $S$ with grade $d$ such that ${\cal F}_{Q+S} \in \bun^c({\cal F}_{P})$. Thus ${\cal F}_{Q}  \in \overline{\bun^c({\cal F}_{P}) \cap \GSYL}$, and this implies $\overline{\bun^c}({\cal F}_{P}) \cap \GSYL \subseteq\overline{\bun^c({\cal F}_{P}) \cap \GSYL}$.  Since $\overline{\bun^c({\cal F}_{P}) \cap \GSYL} \subseteq \overline{\bun^c}({\cal F}_{P}) \cap \GSYL$, we have that $\overline{\bun^c({\cal F}_{P}) \cap \GSYL}= \overline{\bun^c}({\cal F}_{P}) \cap \GSYL$, and the result is proved.
\end{proof}

Now we are ready to state and prove the main result of this paper. It shows that the generic eigenstructures in Theorem \ref{th:improvedDeDo} which are compatible with eigenstructures of linearizations \eqref{linform} of symmetric matrix polynomials with odd grade $d$ and bounded rank give, just by subtracting $(d-1)/2$ to each minimal index, the generic eigenstructures for such polynomials.

\begin{theorem} \label{mainth}
Let $n,r$, and $d$ be integers such that $n \geq 2$, $d \geq 1$ is odd, and $1 \leq r \leq (n-1)${, and let ${\mathbf K}_a$,
for $a = 0,1,\dots,\lfloor \frac{rd}{2} \rfloor$, be the complete eigenstructures of symmetric matrix polynomials $K_a (\la)$ of degree exactly $d$ and rank exactly $r$ in Lemma {\rm\ref{lem.eiglinz}}}. Then,
\begin{itemize}
\item[{\rm(i)}]  For every $n \times n$ complex symmetric matrix polynomial $M(\la)$ of grade $d$ {and} rank at most $r$, there {is some $0\leq a\leq \lfloor\frac{rd}{2}\rfloor$} such that $\overline{\bun^{ s}}(K_a)\supseteq \overline{\bun^{ s}}(M)$.

\item[{\rm(ii)}]  {$\overline{\bun^{ s}}(K_{a}) \bigcap \bun^{ s}(K_{a'})= \emptyset$} 
whenever $a\neq a'$. 

\item[{\rm(iii)}]  The set of $n \times n$ complex symmetric matrix polynomials of grade $d$ with rank at most $r$ is a closed subset of ${\POL_{d, n\times n}^{s}}$ equal to $\bigcup_{0 \leq a \leq \lfloor \frac{rd}{2} \rfloor} \overline{\bun^{ s}}(K_a)$.
\end{itemize}
\end{theorem}

\begin{proof}
(i) For every $n \times n$ symmetric matrix polynomial $M$ of grade $d$ and rank at most $r$, the linearization ${\cal F}_{M}$ {in \eqref{linform}} has rank at most $r+n(d-1)$, because ${\cal F}_{M}$ is unimodularly equivalent to $M \oplus I_{n(d-1)}$.  Thus, for each such $M$ we have that $\overline{{\bun}^c}(\mathcal{K}_{a_1})\supseteq \overline{{\bun}^c}({\cal F}_{M})$, where $\mathcal{K}_{a_1}(\lambda)$ is one of the $nd \times nd$ symmetric matrix pencils of rank $r+n(d-1)$,
defined in Theorem \ref{th:improvedDeDo}.
This means, in particular, that there exists a sequence $\{ y_i \} \subset {\bun}^c(\mathcal{K}_{a_1})$ {converging to} ${\cal F}_{M}$ and, so, for any $i$ large enough, $y_i$ is a small perturbation of ${\cal F}_{M}$ and Theorem~\ref{deform} can be applied to the polynomial $M$ and $y_i$.
Therefore, for these large enough $i$, $y_i$ is congruent to ${\cal F}_{P_i}$ for a certain symmetric polynomial $P_i$ of grade $d$ and size $n\times n$.
The pencils ${\cal F}_{P_i}$ may belong to different orbits (they may differ from each other in the values of the eigenvalues) but all ${\cal F}_{P_i}$ belong to the same bundle ${\bun}^c(\mathcal{K}_{a_1})$ and, therefore, ${\bun}^c(\mathcal{K}_{a_1}) = {\bun}^c({\cal F}_{P_i})$, as well as ${\cal F}_{P_i}$ has rank $r+n(d-1)$, which is equivalent to say that $P_i$ has rank $r$. Now, since ${\cal K}_{a_1}$ has the same eigenstructure (except, perhaps, for the values of the eigenvalues) that ${\cal F}_{P_i}$, then parts (ii)--(iii) in Lemma \ref{lem.eiglinz} guarantee that, for some $0\leq a \leq\lfloor\frac{rd}{2}\rfloor$, $\bun^c({\cal F}_{K_a})=\bun^c({\cal K}_{a_1})=\bun^c({\cal F}_{P_i})$.
Thus $\overline{{\bun}^c}({\cal F}_{{K_a}})\supseteq \overline{{\bun}^c}({\cal F}_{M})$ and $\overline{{\bun}^c}({\cal F}_{{K_a}}) \cap \GSYL \supseteq \overline{{\bun}^c}({\cal F}_{M}) \cap \GSYL$. The latter is equivalent to $\overline{{\bun}^{\rm{syl}}}({\cal F}_{{K_a}}) \supseteq \overline{{\bun^{\rm{syl}}}}({\cal F}_{M})$ by Lemma \ref{cl}, which is, in turn, equivalent to $\overline{\bun^{ s}}({ K_a}) \supseteq \overline{\bun^{ s}}(M)$, {by \eqref{equivalence}}.



(ii) A close look at the proof of Theorem \ref{th:improvedDeDo}-(ii) in  \cite[Thm. 3.2]{dtdmydop2019} allows us to see that the stronger result
$\overline{\bun^c}({\cal K}_{a})
\cap \bun^c({\cal K}_{a'}) = \emptyset$, whenever $a \ne
a'$, holds. Combining this fact with Lemma \ref{lem.eiglinz}-(iii), we immediately obtain that $\overline{{\bun}^c}({\cal F}_{K_a})  \cap {\bun}^c({\cal F}_{K_{a'}}) = \emptyset$. Intersecting the latter equality with $\GSYL$ and applying Lemma \ref{cl} results in $({\overline{{\bun}^c}({\cal F}_{K_a})} \cap \GSYL )\ \bigcap \ \left( {\bun}^c({\cal F}_{K_{a'}}) \cap \GSYL \right) = \emptyset$ and $\overline{{\bun}^{\rm{syl}}}({\cal F}_{K_{a}}) \bigcap \bun^{\rm{syl}}({\cal F}_{K_{a'}})= \emptyset$. This implies $f^{-1}\left(\overline{\bun^{\rm{syl}}}({\cal F}_{K_a})\right)\bigcap f^{-1}(\bun^{\rm{syl}}({\cal F}_{K_a'}))=\emptyset$, with $f$ as in \eqref{homeo}, namely $\overline{\bun^{ s}}(K_{a}) \bigcap \bun^{ s}(K_{a'})= \emptyset$.


(iii) By (i), any $n\times n$ complex symmetric matrix polynomial of grade $d$ with rank at most $r$ is in one of the $\left\lfloor \frac{rd}{2} \right\rfloor+1$ closed sets $\overline{\bun^{ s}}(K_a)$. Thus (iii) holds, since the union of a finite number of closed sets is also a closed set.
\end{proof}

\subsection{Codimension of the generic bundles}

For any $P \in \POL_{d, n\times n}^{s}$, we define the codimension of $\orb^{ s} (P) \subset \POL_{d, n\times n}^{s}$
to be
$
\cod \orb^{ s} (P) :=
\cod \orb^{\rm{syl}}({\cal F}_{P}).$
Translating the arguments in \cite[Sect.~6]{Dmyt15} from skew-symmetric to symmetric matrix polynomials, we see that $\orb^{\rm{syl}}({\cal F}_{P})$ is a manifold in the space of symmetric matrix pencils $\PEN_{n_1 \times n_1}^{s}$, where $n_1 = nd$, and
$
\cod \orb^{\rm{syl}}({\cal F}_{P})=
\cod \orb^{c}({\cal F}_{P}),
$
where the codimension of $\orb^{\rm{syl}}({\cal F}_{P})$ is considered in the space $\GSYL_{d, n\times n}^{s}$ and the codimension of $\orb^{c}({\cal F}_{P})$ in $\PEN_{n_1 \times n_1}^{s}$. The codimensions of $\orb^{c}({\cal F}_{P})$ are computed explicitly via the eigenstructure of ${\cal F}_{P}$ in \cite{DmKS14},
taking into account that congruence orbits are {differentiable} manifolds in the complex $n_1^2+n_1$ dimensional space of symmetric pencils and, so, their codimensions are well defined as the dimensions of their normal spaces. Therefore, for the generic $n\times n$ symmetric matrix polynomial $K_a$ with grade $d$ identified in Theorem \ref{mainth}, we have
$\cod \orb^{ s}(K_a) = \cod \orb^{c}({\cal F}_{K_a}) = \cod \orb^{c} ({\cal K}_{a_1})$, where ${\cal K}_{a_1}$ is the pencil in Theorem~\ref{th:improvedDeDo} with the identifications $n \mapsto n_1=nd$, $r \mapsto r_1 = n(d-1) + r$, and $a \mapsto a_1=a+\frac{1}{2}(n-r)(d-1)$. Therefore, from \cite[Sect.~5]{dtdmydop2019} we obtain:
\begin{align*}
\cod \orb^{ s}(K_a) & = (n_1-a_1)(n_1 - r_1 + 1) \\
	     &= \left(nd - \frac{1}{2}(n(d-1)+r-rd) -a\right)(nd - (n(d-1)+r) +1)\\
             & = \frac{1}{2} (n(d+1) + r(d-1) -2a) (n - r +1) \\
             &= \frac{1}{2} \left((n+ r)(d-1) +2(n-a)\right) (n - r +1)  \, .
\end{align*}
Following, e.g., \cite{Dmyt15,DmKa14,EdEK99,JoKV13}, {we set} the codimension of $\bun^{ s}(P)$, for a matrix polynomial $P$, as:
\begin{equation*} \label{buncodim}
    \cod  \bun^{ s}(P):=
    \cod  \orb^{ s} (P) - \ \#
    \left\{ \text{distinct eigenvalues of } {P} \right\}.
\end{equation*}
Therefore, for the generic bundles $\bun^{ s}(K_a)$, we have:
\begin{equation*} \label{codimcompbun}
\begin{aligned}
\cod \bun^{ s}(K_{a}) &= \frac{1}{2} \left((n+ r)(d-1) +2(n-a)\right) (n - r +1) - rd + 2a \\
&= \frac{1}{2} (n+ r)(d-1) - rd + n + n(n - r) - a (n - r - 1)    \\
&= \frac{1}{2} (nd - rd + n - r) + n(n - r) - a (n - r - 1)    \\
& = (n-r)\left(n + \frac{d+1}{2}\right)-a(n-r-1).
\end{aligned}
\end{equation*}

Note that the generic bundles in Theorem \ref{mainth} have different codimensions if $r<n-1$, because in this case the codimensions depend on the parameter $a$. In particular, the bundle with the largest $a$ has the smallest codimension or, equivalently, the largest dimension. In other words, the bundle with less eigenvalues has the largest dimension.

\bigskip

\section{Openness of the generic bundles}  \label{sec.openbunbles}
In order to express the results in this section concisely, we introduce the following additional notation: $\PEN_{n\times n}^s(r)$ (respectively, ${\POL_{d, n\times n}^{s}}(r)$) denotes the set of $n\times n$ symmetric pencils (resp., matrix polynomials of grade $d$) with rank at most $r$.

Theorem \ref{mainth}-(iii) shows that the complete eigenstructures ${{\mathbf K}_a}$ and their associated bundles $\bun^{ s} (K_a)$ play a prominent role in ${\POL_{d, n\times n}^{s}}(r)$, because any $n\times n$ symmetric matrix polynomial of grade $d$ and rank at most $r$ is the limit of a sequence of polynomials having exactly one of these structures. In other words, the set ${\POL_{d, n\times n}^{s}}(r)$ can be decomposed into a finite number of pieces where each of these bundles is dense. Equivalently,
$${\POL_{d, n\times n}^{s}}(r) = \overline{\bigcup_{0 \leq a \leq \lfloor \frac{rd}{2} \rfloor} \bun^{ s} (K_a)}$$
holds, since the closure of the union of a finite number of sets is equal to the union of their closures. Therefore, we can state that $\bigcup_{0 \leq a \leq \lfloor \frac{rd}{2} \rfloor} \bun^{ s} (K_a)$ is dense in ${\POL_{d, n\times n}^{s}}(r)$. However, this result does not guarantee that $\bigcup_{0 \leq a \leq \lfloor \frac{rd}{2} \rfloor} \bun^{ s} (K_a)$ is ``generic" inside ${\POL_{d, n\times n}^{s}}(r)$ in the standard topological sense, since we do not know yet whether or not $\bigcup_{0 \leq a \leq \lfloor \frac{rd}{2} \rfloor} \bun^{ s} (K_a)$ is open in ${\POL_{d, n\times n}^{s}}(r)$. The goal of this section is to prove this property. In fact, Theorem \ref{th:open} proves a stronger result, namely, that each bundle $\bun^{ s} (K_a)$ is open in ${\POL_{d, n\times n}^{s}}(r)$. So, the eigenstructures ${{\mathbf K}_a}$ in \eqref{kcilist} can be properly termed in this strong topological sense as the ``generic" complete eigenstructures of $n\times n$ symmetric matrix polynomials with grade $d$ and rank at most $r$. Theorem \ref{th:open} provides an analytical formulation of this openness in terms of the distance introduced in Section \ref{sect.linz} in the spirit of \cite[Cor. 3.3]{dmydop-laa-2017}. However, we emphasize that the fact that we are dealing with bundles instead of with orbits makes the proof of Theorem \ref{th:open} considerably more difficult than the proof of \cite[Cor. 3.3]{dmydop-laa-2017}. This point will be further discussed at the end of this section.

We first state and prove the corresponding result for matrix pencils, on which the result for matrix polynomials of arbitrary odd grade $d$ is based.

\begin{theorem}\label{th:open-pencil} {\rm (Analytical formulation of generic bundles in $\PEN_{n\times n}^s(r)$).}
Let $n$ and $r$ be integers such that $n \geq 2$ and $1\leq r \leq n-1$, and let ${\cal K}_{a} (\lambda)$, for $a=0,1,\ldots,\lfloor\frac{r}{2}\rfloor\,$, be any of the symmetric $n \times n$ pencils with rank $r$ in Theorem {\rm\ref{th:improvedDeDo}}. Then
\begin{itemize}
\item[\rm (i)] For every $n\times n$ complex symmetric matrix pencil with rank at most $r$, $S$, and every $\epsilon>0$, there is an $n\times n$ complex symmetric matrix pencil $S_a\in \bun^c(\cK_a)$, for some $0\leq a\leq \lfloor\frac{r}{2}\rfloor$, such that $d(S,S_a)<\epsilon$.
\item[\rm (ii)] Let $0\leq a\leq \lfloor\frac{r}{2}\rfloor$. Then, for every $n\times n$ complex symmetric matrix pencil $S\in \bun^c(\cK_a)$, there is a number $\epsilon>0$ such that
\begin{equation}\label{brepsilon}
\cB_r(S,\epsilon):=\left\{S':\ S'\in\PEN_{n\times n}^s(r)\ \mbox{\rm and}\ d(S,S')<\epsilon\right\} \subseteq \bun^c(\cK_a).
\end{equation}
Moveover, such set $\cB_r(S,\epsilon)$ contains infinitely many elements.
\end{itemize}
\end{theorem}

\begin{proof}
Part (i) is an immediate consequence of Theorem \ref{th:improvedDeDo}-(i) combined with the standard definition of closure in a metric space.

Part (ii). By Theorem \ref{th:improvedDeDo}-(iii), $\cB_r(S,\epsilon) \subseteq \bigcup_{a'\neq a}\overline{\bun^c}(\cK_{a'})\cup \overline{\bun^c}(\cK_a)$, for all $\epsilon>0$. There are two possible complementary cases: either $\cB_r(S,\epsilon)$ contains at least one pencil in $\bigcup_{a'\neq a}\overline{\bun^c}(\cK_{a'})$ for all $\epsilon>0$, or not. Let us analyze separately these two cases.

{\bf Case 1}: For all $\epsilon>0$, $\cB_r(S,\epsilon)$ contains at least one pencil in $\bigcup_{a'\neq a}\overline{\bun^c}(\cK_{a'})$. This is equivalent to state that for all $\epsilon>0$, $\cB_r(S,\epsilon)$ contains infinitely many pencils in $\bigcup_{a'\neq a}\overline{\bun^c}(\cK_{a'})$, because $S \notin \bigcup_{a'\neq a}\overline{\bun^c}(\cK_{a'})$ as a consequence of Theorem \ref{mainth}-(ii) applied to pencils (that is with $d=1$). Since the number of sets in the union $\bigcup_{a'\neq a}\overline{\bun^c}(\cK_{a'})$ is finite, for all $\epsilon >0$,  $\cB_r(S,\epsilon)$ contains infinitely many pencils in some of these sets, say $\overline{\bun^c}(\cK_{\widetilde a})$, with $\widetilde a\neq a$. Therefore, $S\in\overline{\overline{\bun^c}(\cK_{\widetilde a})}=\overline{\bun^c}(\cK_{\widetilde a})$, which is a contradiction with the fact that $\bun^c(\cK_{a})\cap \overline{\bun^c(\cK_{\widetilde a})}=\emptyset$ in Theorem \ref{mainth}-(ii) (for $d=1$)\footnote{Though this equality is true for matrix pencils, as a particular case of matrix polynomials, it is not stated in Theorem \ref{th:improvedDeDo} for consistency, since this theorem is written as in the original reference \cite{dtdmydop2019}, where this equality was not included.}. Therefore, Case 1 cannot happen and its negation, i.e. Case 2, must happen.

{\bf Case 2}: There is some $\epsilon_0>0$ such that $\cB_r(S,\epsilon_0)$ does not contain pencils from $\bigcup_{a'\neq a}\overline{\bun^c}(\cK_{a'})$. Then, $\cB_r(S,\epsilon_0)\subseteq\overline{\bun^c}(\cK_a)$. We are going to prove that, in this case, there is some $\epsilon>0$ (with $\epsilon\leq\epsilon_0$) such that $\cB_r(S,\epsilon)\subseteq \bun^c(\cK_a)$. We proceed by contradiction. Assume that for all $0<\epsilon\leq\epsilon_0$, the set $\cB_r(S,\epsilon)$ contains at least one pencil in $\overline{\bun^c}(\cK_a)\setminus \bun^c(\cK_a)$. Then, there is a sequence, $\{S_k\}_{k\in\mathbb N}$, with $S_k\in \overline{\bun^c}(\cK_a)\setminus \bun^c(\cK_a)$, which converges to $S$. Moreover, since $\rank S_k \leq r$ for all $k$ and $\rank S = r$, there is a $k_0$ such that all the pencils in the subsequence $\{S_k\}_{k>k_0}$ have rank exactly $r$ since, otherwise, ${\rm rank}\, S={\rm rank}(\lim S_k)< r$, a contradiction. Therefore, we may assume that all pencils in $\{S_k\}_{k\in\mathbb N}$ have rank $r$. This is equivalent to say that the numbers of left (and right) minimal indices of all pencils $S_k$ are equal to those of $S$ and equal to $n-r$.

Following the notation in \cite{Hoyo90}, given a matrix pencil $H$ with right minimal indices $\varepsilon_1\geq\cdots\geq\varepsilon_p$ and left minimal indices $\eta_1\geq\cdots\geq\eta_q$, we denote by ${\mathfrak r}(H)=(r_1(H),r_2(H),\hdots)$ and ${\mathfrak l}(H)=(\ell_1(H),\ell_2(H),\hdots)$ the sequences whose $i$th terms are $r_i(H)={\rm card}\{j:\varepsilon_j\geq i\}$ and $\ell_i(H)={\rm card}\{j:\eta_j\geq i\}$, respectively. Now, since, for each $k$, $S_k\in\overline{\bun^c}(\cK_a)$, we have that each $S_k$ is the limit of a sequence of pencils in $\bun^c(\cK_a)$, which have all the same minimal indices and rank $r$ as $S$. Thus, Lemma 1.2 in \cite{Hoyo90} can be applied to each of these sequences and implies that,

\begin{itemize}
\item[(i)] ${\mathfrak r}(S_k)\prec\prec {\mathfrak r}(S)$, and
\item[(ii)] ${\mathfrak l}(S_k)\prec\prec {\mathfrak l}(S)$,
\end{itemize}
where $a\prec\prec b$ means weak majorization of nonincreasing sequences (see \cite[\S1]{Hoyo90}). Besides this, and since $\lim S_k=S$, Lemma 1.2 in \cite{Hoyo90} again implies that, for $k$ large enough,
\begin{itemize}
\item[(i')] ${\mathfrak r}(S)\prec\prec {\mathfrak r}(S_k)$, and
\item[(ii')] ${\mathfrak l}(S)\prec\prec {\mathfrak l}(S_k)$,
\end{itemize}
Now, (i), (ii), (i'), and (ii'), together with the fact that weak majorization of nonincreasing sequences is an order relation, imply that
$$
{\mathfrak r}(S_k)={\mathfrak r}(S),\qquad\mbox{and}\qquad {\mathfrak l}(S_k)={\mathfrak l}(S),
$$
for $k$ large enough. This implies that, for $k$ large enough, $S_k$ and $S$ have the same left and right minimal indices. Moreover, recall that the sum of all left and right minimal indices of a given pencil $H$, together with the sum of all partial multiplicities (finite and infinite) of $H$ add up to the rank of $H$ (this is just a consequence of the Kronecker Canonical Form, but we also refer the reader to \cite[Lemma 6.3]{DeDM14}). Therefore, for $k$ large enough, the sum of all partial multiplicities of $S_k$ (finite and infinite) is also the same as that of $S$, which is equal to $r-2a$.

Next, we are going to see that all the $r-2a$ (counting multiplicities) eigenvalues of $S_k$, for $k$ large enough, are simple. This is a consequence of the fact that all eigenvalues, $\mu_1 , \ldots , \mu_{r-2a}$, of $S$ are simple, combined with conditions (iii)-(iv) in \cite[Lemma 1.2]{Hoyo90} applied to $\lim S_k=S$. More precisely, first, since the $r-2a$ eigenvalues of $S$ are simple, condition (iii) in \cite[Lemma 1.2]{Hoyo90} implies that, for $k$ large enough, the $r-2a$ (counting multiplicities) eigenvalues of $S_k$ are included in the union of $r-2a$ pairwise disjoint neighbourhoods, $B(\mu_1,\delta) , \ldots , B(\mu_{r-2a},\delta)$, of the eigenvalues of $S$. Second, again since the $r-2a$ eigenvalues of $S$ are simple, the right hand side of the inequality (iv) in \cite[Lemma 1.2]{Hoyo90} is equal to $1$, for each eigenvalue $\mu_i$ of $S$ (in other words, the sum of all partial multiplicities of $\mu_i$ is equal to $1$, since $\mu_i$ is simple). Therefore, the inequality (iv) in \cite[Lemma 1.2]{Hoyo90} implies that the sum of the partial multiplicities of all eigenvalues of $S_k$ (for $k$ large enough) in $B(\mu_i,\delta)$ is at most $1$, for $i =1,\ldots,r-2a$. But, it must be equal to $1$, since the $r-2a$ eigenvalues of $S_k$ (for $k$ large enough) are included in  $B(\mu_1,\delta) \cup \cdots \cup B(\mu_{r-2a},\delta)$ and these neighbourhoods are pairwise disjoint. So, all the $r-2a$ eigenvalues of $S_k$ are simple, for $k$ large enough.

Summarizing, for $k$ large enough, $S_k$ has $r-2a$ simple eigenvalues, and the same left and right minimal indices as $S$. Hence, $S_k\in B^c(\cK_a)$, a contradiction with $S_k\in \overline{\bun^c}(\cK_a)\setminus \bun^c(\cK_a)$.

In order to prove that $\cB_r(S,\epsilon)$ contains infinitely many elements, note that for any $n\times n$ complex constant matrix $E$ and for any sufficiently small complex number $z \ne 0$, $(I + z E)^\top S (I + z E) \in \cB_r(S,\epsilon)$.
\end{proof}


Next, we state and prove the analog of Theorem \ref{th:open-pencil} for complex symmetric matrix polynomials of odd grade $d$ and rank at most $r$. In our opinion, this is the second main result in this paper behind Theorem \ref{mainth}.

\begin{theorem}\label{th:open} {\rm (Analytical formulation of generic bundles in $\POL_{d,n\times n}^{s}(r)$).} Let $n,r$, and $d$ be integers such that $n \geq 2$, $d\geq 1$ is odd, and $1 \leq r \leq n -1$, and let $K_a(\la)$, for $a=0,1,\ldots,\lfloor\frac{rd}{2}\rfloor\,$, be any of the symmetric $n\times n$ matrix polynomials of degree exactly $d$ and rank exactly $r$ in Theorem \ref{mainth}. Then
\begin{enumerate}
\item[\rm (i)] For every $n \times n$ complex symmetric matrix polynomial $P$ of grade $d$ with rank at most $r$ and every $\epsilon > 0$, there exists an $n \times n$ complex symmetric matrix polynomial $P_a \in \bun^s (K_a)$, for some $0 \leq a \leq \left\lfloor \frac{rd}{2} \right\rfloor$, such that $d(P,P_a) < \epsilon$.

\item[\rm (ii)] Let $0\leq a \leq \left\lfloor \frac{rd}{2} \right\rfloor$. Then, for every $n\times n$ complex symmetric matrix polynomial $P\in \bun^s (K_a)$, there exists a number $\epsilon > 0$ such that
    \begin{equation}\label{brepsilon-pol}
    \mathcal{B}_r (P, \epsilon) := \left\{ P' \, : \, P' \in \POL_{d,n\times n}^{s}(r) \, \, \mbox{\rm and} \, \, d(P,P') < \epsilon \right\} \subseteq \bun^s (K_a).
    \end{equation}
\end{enumerate}
Moreover, such set $\mathcal{B}_r (P, \epsilon)$ contains infinitely many elements.
\end{theorem}

\begin{proof}
Part (i) follows immediately from Theorem \ref{mainth}-(i) combined with the standard definition of closure in a metric space.

To prove (ii), let $P\in \bun^s(K_a)$. Note that, by Lemma \ref{lem.eiglinz}-(iii), ${\cal F}_P \in \bun^c ({\cal F}_P) = \bun^c ({\cal F}_{K_a}) = \bun^c(\cK_{a_1})$, with $a_1=a+\frac{1}{2}(n-r)(d-1)$, and ${\rm rank}\,{\cal F}_P= {\rm rank}\,{\cal F}_{K_a} = r_1:=r+n(d-1)$. By Theorem \ref{th:open-pencil}-(ii), there is some $\epsilon >0$ such that ${\cal B}_{r_1}({\cal F}_P,\epsilon)\subseteq \bun^c(\cK_{a_1}) = \bun^c ({\cal F}_{K_a})$, with ${\cal B}_{r_1}({\cal F}_P,\epsilon)$ defined as in \eqref{brepsilon}. Therefore, ${\cal B}_{r_1}({\cal F}_P,\epsilon)\cap \GSYL^{s}_{d, n\times n}\subseteq \bun^c({\cal F}_{K_a})\cap\GSYL^{s}_{d, n\times n}=\bun^{\rm syl}({\cal F}_{K_a})$. Then,
$$
f^{-1}({\cal B}_{r_1}({\cal F}_P,\epsilon)\cap \GSYL^{s}_{d, n\times n})\subseteq f^{-1}(\bun^{\rm syl}({\cal F}_{K_a}))=\bun^s(K_a),
$$
where $f$ is the bijective isometry in \eqref{homeo}. To conclude the proof of the statement, we just need to see that
\begin{equation}\label{identity-f}
f^{-1}({\cal B}_{r_1}({\cal F}_P,\epsilon)\cap \GSYL^{s}_{d, n\times n})={\cal B}_r(P,\epsilon),
\end{equation}
with ${\cal B}_r(P,\epsilon)$ defined as in \eqref{brepsilon-pol}. For this, we just recall that for any $Q \in \POL_{d,n\times n}^{s}$
\begin{itemize}
\item[(a)] $d(P,Q)<\epsilon$ if and only if $d({\cal F}_P,{\cal F}_Q)<\epsilon$, since $f$ is an isometry, and

\item[(b)] ${\rm rank} \,{\cal F}_Q= {\rm rank} \,Q + n (d-1)$.
\end{itemize}

Facts (a) and (b) together imply that, if $Q\in f^{-1}({\cal B}_{r_1}({\cal F}_P,\epsilon)\cap \GSYL^{s}_{d, n\times n})$, then $d({\cal F}_P,{\cal F}_Q)<\epsilon$ and ${\rm rank}\,{\cal F}_Q\leq r_1$, so $d(P,Q)<\epsilon$ and ${\rm rank}\,Q\leq r$, that is, $Q\in {\cal B}_r(P,\epsilon)$. Conversely, if $Q\in {\cal B}_r(P,\epsilon)$, then $d({\cal F}_P,{\cal F}_Q)<\epsilon$ and ${\rm rank}\,{\cal F}_Q\leq r_1$, so $Q\in  f^{-1}({\cal B}_{r_1}({\cal F}_P,\epsilon)\cap \GSYL^{s}_{d, n\times n})$. Hence, \eqref{identity-f} is proved.

Finally, in order to prove that $\cB_r(P,\epsilon)$ contains infinitely many elements, we proceed as in Theorem \ref{th:open-pencil} and simply note that, for any $n\times n$ complex constant matrix $E$ and for any sufficiently small complex number $z \ne 0$, $(I + z E)^\top P (I + z E) \in \cB_r(P,\epsilon)$.
\end{proof}

As an immediate corollary of Theorem \ref{th:open}, the fact that the union of open sets is an open set, and the fact that the closure of the union of a finite number of sets is the union of their closures, we state, for completeness, the following result, using the same notation as in Theorem \ref{th:open}.

\begin{corollary}
The set $\bigcup_{0 \leq a \leq \lfloor \frac{rd}{2} \rfloor} \bun^{ s} (K_a)$ is open and dense in ${\POL_{d, n\times n}^{s}}(r)$.
\end{corollary}

Theorem \ref{th:open-pencil} (resp., Theorem \ref{th:open}) shows that the bundles $\bun^c(\cK_a)$ (resp., $\bun^s(K_a)$) are open in $\PEN_{n\times n}^s(r)$ (resp., ${\POL_{d, n\times n}^{s}}(r)$). Therefore, in particular, they are open in their closures. The property of being open in their closures is well-known for congruence orbits of arbitrary symmetric pencils (as well as orbits under strict equivalence of arbitrary unstructured pencils) \cite[Closed Orbit Lemma, p. 53]{Bore91}. A straightforward argument using the isometry $f$ in \eqref{homeo} also proves that the orbits $\orb^s(P)$ of symmetric matrix polynomials of odd grade $d$ are open in their closures. However, up to our knowledge, it is not yet known whether the bundles $\bun^c(S)$, for an arbitrary symmetric pencil $S$, and also $\bun^s(P)$, for an arbitrary symmetric polynomial $P$ of grade $d$, are open in their closures. In fact, the same questions remain open for arbitrary unstructured pencils and bundles under strict equivalence, and for bundles of unstructured matrix  polynomials of fixed grade. Theorems \ref{th:open-pencil} and \ref{th:open} show, however, that the generic bundles $\bun^c(\cK_a)$ and $\bun^s(K_a)$ are open in their closures.

We finish this section with a comparison between the simple and short proof of the analog result to Theorem \ref{th:open} for unstructured matrix polynomials of fixed grade and bounded rank, i.e., \cite[Cor. 3.3]{dmydop-laa-2017}, and the more complicated and longer proofs of Theorems \ref{th:open-pencil} and \ref{th:open}. The key difference is that \cite[Cor. 3.3]{dmydop-laa-2017} deals with orbits. The strategy of the proof of \cite[Cor. 3.3]{dmydop-laa-2017}, which is only very briefly sketched in \cite{dmydop-laa-2017}, is to prove first the result for pencils following exactly the same approach as in the proof of Theorem \ref{th:open-pencil}. However, in the case of \cite[Cor. 3.3]{dmydop-laa-2017}, the solution of Case 2 is trivial just by recalling that the corresponding orbit is open in its closure. The proof for polynomials of grade higher than one in \cite[Cor. 3.3]{dmydop-laa-2017} is completely similar to the one of Theorem \ref{th:open}. We emphasize that in stark contrast with \cite{dmydop-laa-2017}, the solution of Case 2 in the proof of Theorem \ref{th:open-pencil} is far from trivial as a consequence of the lack of a result guaranteeing that bundles are open in their closures.

\section{Conclusions and future work}\label{sec.conclusion}
In this paper, we have obtained the generic eigenstructures of $n\times n$ complex symmetric matrix polynomials with bounded (deficient) rank and fixed odd grade. More precisely, we have seen (Theorems \ref{mainth}) that the set of $n\times n$ symmetric matrix polynomials with rank at most $r<n$ and odd grade $d$ is the union of the closures of $\lfloor\frac{rd}{2}\rfloor+1$ subsets corresponding to $n\times n$ symmetric matrix polynomials with degree exactly $d$ and rank exactly $r$ having some specific eigenstructures, up to the values of the eigenvalues. Moreover, such subsets are open in the set of $n\times n$ symmetric matrix polynomials with rank at most $r$ and odd grade $d$ (Theorem \ref{th:open}). Therefore, the eigenstructures described in Theorem \ref{mainth} are the generic ones. In order to prove this main result, we have shown that, given any list of $s$ scalar polynomials of degree $1$ with different roots, an integer $t\in\{0,1\}$, together with a list of $n-r$ nonnegative integers, there exists an $n\times n$ symmetric matrix polynomial with rank $r$ and grade $d$ whose finite elementary divisors are the ones in the given list of scalar polynomials, with $t$ infinite linear elementary divisors, and whose left (and, consequently, right) minimal indices are the given $n-r$ nonnegative integers, if and only if the sum of $s,t$, and twice the sum of the nonnegative integers equals $rd$.

Results in the spirit of this paper have been recently obtained in \cite{dmydop-laa-2017,dmydop-laa-2018} for, respectively, general (unstructured) $m\times n$ matrix polynomials with bounded deficient rank and given grade, and $n\times n$ skew-symmetric matrix polynomials of bounded deficient rank and given odd grade. However, we emphasize that the symmetric generic eigenstructures identified in this paper are very different from the eigenstructures found in \cite{dmydop-laa-2017,dmydop-laa-2018}, since the symmetric ones involve eigenvalues while the others not.

A natural continuation of the research in this paper and in \cite{dmydop-laa-2017,dmydop-laa-2018} consists of trying to determine the generic eigenstructures of the sets of other structured $n\times n$ matrix polynomials with bounded rank and given grade, like the $\top$-palindromic and $\top$-alternating ones. It is also natural to try to extend the results in this paper to symmetric matrix polynomials with even grade. However, this will require different tools and techniques since, to begin with, there is no a symmetric linearization like ${\cal F}_P$ in \eqref{linform} for symmetric matrix polynomials with even grade.



\bibliographystyle{abbrv}

\end{document}